\documentclass[reqno]{amsart}
\usepackage[T1]{fontenc}
\usepackage[latin1]{inputenc}
\usepackage[english]{babel}
\usepackage[babel]{csquotes}

\usepackage{cite}
\usepackage{amssymb}
\usepackage{amsmath}
\usepackage{amsthm}
\usepackage{latexsym}
\usepackage{graphicx}
\usepackage{mathrsfs}
\usepackage{mathrsfs}
\usepackage{bbm}
\usepackage{verbatim}
\usepackage[colorlinks=true, pdfstartview=FitV, linkcolor=blue, 
citecolor=blue, urlcolor=blue]{hyperref}
\numberwithin{equation}{section}

\usepackage[usenames,dvipsnames]{color}

\newtheorem{thm}{Theorem}[section]

\newtheorem{lem}[thm]{Lemma}
\newtheorem{prop}[thm]{Proposition}
\newtheorem{defin}[thm]{Definition}
\newtheorem{remark}[thm]{Remark}

\def\enne{\mathbb{N}}

\def\erre{\mathbb{R}}
\def\Rz{\mathbb{R}}
\def\Nz{\mathbb{N}}

\def\P{\mathbb{P}}

\def\E{\mathop{{}\mathbb{E}}}

\def\L{L}
\def\C{K}

\def\cL{\mathscr{L}}
\def\cF{\mathscr{F}}
\def\cB{\mathscr{B}}
\def\eps{\varepsilon}
\def\cP{\mathscr{P}}
\def\mP{\mathcal{P}}

\renewcommand{\d}{{\mathrm d}}

\def\beq{\begin{equation}}
\def\eeq{\end{equation}}

\def\to{\rightarrow}
\def\wto{\rightharpoonup}
\def\wstarto{\stackrel{*}{\rightharpoonup}}
\def\embed{\hookrightarrow}

\def\norm #1{\left\|#1\right\|}

\def\sp #1#2{\left<#1,#2\right>}
\newcommand\ip\sp

\begin{document}
\title[Weak stability by noise for doubly nonlinear evolution equations]
{Weak stability by noise 
for approximations of doubly nonlinear evolution equations}

\author{Carlo Orrieri}
\address[Carlo Orrieri]{Department of Mathematics, 
Universit\`a di Pavia, Via Ferrata 1, 27100 Pavia, Italy.}
\email{carlo.orrieri@unipv.it}
\urladdr{http://www-dimat.unipv.it/orrieri}

\author{Luca Scarpa}
\address[Luca Scarpa]{Department of Mathematics, Politecnico di Milano, 
Via E.~Bonardi 9, 20133 Milano, Italy.}
\email{luca.scarpa@polimi.it}
\urladdr{https://sites.google.com/view/lucascarpa}

\author{Ulisse Stefanelli}
\address[Ulisse Stefanelli]{Faculty of Mathematics,  University of Vienna, Oskar-Morgenstern-Platz 1, A-1090 Vienna, Austria,
Vienna Research Platform on Accelerating Photoreaction Discovery,  $\&$
University of Vienna, W\"ahringer Str. 17, 1090 Vienna, Austria}
\email{ulisse.stefanelli@univie.ac.at}
\urladdr{http://www.mat.univie.ac.at/$\sim$stefanelli}



\begin{abstract}
  Doubly nonlinear stochastic evolution equations are
  considered. Upon assuming the additive noise to be rough enough, we prove the existence of
  probabilistically weak solutions of Friedrichs type and study their
  uniqueness in law. This entails stability for approximations of 
  stochastic doubly nonlinear equations in a weak probabilistic sense.
  Such effect is a genuinely stochastic, as doubly
  nonlinear equations are not even expected to exhibit 
  uniqueness in the deterministic case. 
\end{abstract}

\maketitle


\section{Introduction}
\setcounter{equation}{0}
\label{sec:intro}

We are interested in abstract doubly nonlinear evolution equations in the form 
\beq\label{eq0}
  A(\partial_t u) + B(u) \ni F(u)\,, \quad u(0)=u_0\,,
\eeq
where the nonlinearities $A$ and $B$ are maximal monotone operators on a Hilbert space $H$,
$F:H\to H$ is a Lipschitz-continuous forcing, and $u_0$ is a given
initial datum. Solutions $u:[0,T]\to H$ to \eqref{eq0} may
describe the dynamics of a given system, up to the final reference
time $T>0$. The term $B$ represents the (sub-) differential of the 
energy driving the system, whereas $A$ corresponds to the dissipation.
In typical examples, $B$ is a nonlinear differential operator while $A$
is a nonlinear map, associated to a possibly multivalued graph.
Partial differential equations in doubly nonlinear form ubiquitously arise in applications 
to classical thermodynamics, e.g.~the well known two-phase Stefan problem,
to mechanics, e.g.~the quasistatic limit of a damped oscillator,
to electronics, e.g.~in RC circuits, and viscoelastoplasticity:
we refer to \cite[Sec.~6]{SS-doubly2} for a brief survey of such examples.

The mathematical literature on deterministic doubly nonlinear
evolutions is extensive.   The special case $A
= {\rm id} $  
corresponds to    perturbed gradient flows, and is rather classical.
The genuine doubly nonlinear case has been originally tackled in
\cite{Colli, Colli-Visintin, diben-show}, where existence of solutions was investigated.
Solutions to \eqref{eq0} are however not expected to be unique, see
\cite{Akagi10}. Some sufficient conditions for uniqueness can be found
in   \cite{Colli-Visintin, diben-show}.

As a consequence of nonuniqueness, the convergence of
approximations of deterministic doubly nonlinear equations may be restricted
to subsequences only. 
More precisely, consider a family of approximated problems 
\beq\label{eq0'}
  A(\partial_t u_\lambda) + B_\lambda(u_\lambda) = F(u_\lambda)\,, \quad u_\lambda(0)=u_0^\lambda\,,
\eeq
where $\lambda>0$ is a positive parameter, $(B_\lambda)_\lambda$ is a 
family of single-valued approximations of $B$, and $(u_0^\lambda)_\lambda$
is a family of regularisations of the initial datum $u_0$. As the
limiting problem \eqref{eq0'}  for $\lambda \to 0$ may have multiple solutions,
convergence of the full sequence $(u_\lambda)_\lambda$ cannot be
directly guaranteed, 
as different subsequences may have different limits. 

The purpose of this paper is to show that the addition of a suitable noise 
may restore stability for such approximations, in a weak probabilistic formulation.
The stochastic counterpart of doubly nonlinear equations \eqref{eq0} was firstly 
introduced in \cite{SS-doubly2} and reads
\beq
  \label{eq0_st}
  \begin{cases}
  \d u =(\partial_t u^d)\,\d t + G\,\d W\,,\\
  A(\partial_t u^d) + B(u) \ni F(u)\,,\\
  u(0)=u_0\,,
\end{cases}
\eeq
where $W$ is a $H$-cylindrical Wiener process and $G$ is a
Hilbert--Schmidt covariance operator. In \cite{SS-doubly2},
existence of solutions to \eqref{eq0_st} has been obtained in the
in the probabilistically weak and
analytically strong sense. 
Stochastic doubly nonlinear equations were also investigated in 
\cite{Barbu-DaPrato, SWZ18} in the context of the two-phase stochastic Stefan problem,
in \cite{ScarStef-SDNL} in a more abstract framework, 
and in \cite{SS-rate} in the rate-independent case.

The stochastic version of the approximated problem \eqref{eq0'} reads
\beq
  \label{eq0'_st}
  \begin{cases}
  \d u_\lambda =(\partial_t u_\lambda^d)\,\d t + G_\lambda\,\d W\,,\\
  A(\partial_t u_\lambda^d) + B_\lambda(u_\lambda) = F(u_\lambda)\,,\\
  u_\lambda(0)=u_0^\lambda\,,
\end{cases}
\eeq
where $(G_\lambda)_\lambda$ are possibly regularised operators approximating $G$.
Under suitable assumptions on the approximating families $(B_\lambda)_\lambda$, 
$(u_0^\lambda)_\lambda$, and $(G_\lambda)_\lambda$, we prove that the 
addition of noise stabilises the dynamics of the approximated problem \eqref{eq0'_st}.
More precisely, we are able to show that, given the approximating sequence,
there exists a unique probability measure $\mu$ on the space of trajectories of \eqref{eq0'}
such that the probability distributions of any solution to \eqref{eq0'_st}
weakly converge to $\mu$ as $\lambda\to0$.

Let us briefly highlight the challenges of the analysis, as well
as the technical strategy for overcome them. 
Existence of solutions for the stochastic system \eqref{eq0_st}
was proved in \cite{SS-doubly2} in the analytically-strong sense
under the assumption that the noise coefficient $G$ is
coloured enough, precisely that $G$ is Hilbert--Schmidt from $H$ to $V$, where
$V$ is the natural space associated to the weak formulation of $B:V\to V^*$.
Let us recall that for doubly nonlinear problems, the analytically-strong 
formulation is usually the only one available: indeed, since generally 
the operator $A$ is of order zero on $H$, if one gives sense to the nonlinearity 
$A(\partial_t u^d)$, then by comparison in \eqref{eq0_st} it holds that $B(u)\in H$
at least almost everywhere, so that the variable $u$
is required to takes values in the domain $D(B)$ of $B$. In particular, 
weaker notions of solutions (e.g.~variational ones) are not directly available
and existence of (analytically strong) solutions can be expected only 
under the assumption that $G$ is Hilbert--Schmidt with value in $V$.

The main technical idea is to rely on techniques from the research stream on
regularisation by noise, namely, to
exploit the Kolmogorov equation associated to 
\eqref{eq0_st}.  Without any claim of completeness, we refer to  
\cite{AB23, BOS_uniq, DPF, Dap1, DPFPR2, priola2, priola2_corr}
and the references therein for the recent literature on 
uniqueness by noise for infinite dimensional stochastic evolution equations.

Due to the unusual form of the doubly nonlinear equation,
this calls for reformulating the stochastic 
system \eqref{eq0_st} in a more classical way, namely 
\[
  \d u + B(u)\,\d t = [F(u)+\C(u)]\,\d t + G\,\d W\,, \qquad u(0)=u_0\,,
\]
where the operator $\C:D(B)\to H$ is defined as
\[
  \C(z):= A^{-1}(F(z)-B(z)) - (F(z)-B(z))\,, \quad z\in D(B)\,.
\]
The Kolmogorov equation associated to \eqref{eq0_st}
formally reads 
\begin{align*}
  &\alpha\varphi(x) - \frac12\operatorname{Tr}\left[GG^*D^2\varphi(x)\right]
  +\left( B(x), D\varphi(x)\right)_H \\
  &\qquad= g(x) +\left(F(x) + \C(x), D\varphi(x)\right)_H\,,
  \quad x\in D(B)\,,
\end{align*}
where $\alpha>0$ is fixed and $g:H\to\erre$ is a given forcing term.
We point out that the Kolmogorov equation is extremely pathological due 
to presence of both the perturbation term $\C$, which is of the same order of 
the leading operator $B$, and the coloured coefficient $G\in\cL^2(H,V)$. 
This features makes the first-order perturbation term too singular and 
the second-order diffusion too degenerate, respectively.
Consequently, we do not expect to solve the Kolmogorov equation, 
not even in some suitably  mild fashion.

In order to overcome this obstruction, we 
introduce a weaker notion of solution
to the limiting stochastic equation \eqref{eq0_st}, allowing
for rougher noise coefficients $G$ being 
 only Hilbert--Schmidt with values in $H$. A natural candidate is the concept of
Friedrichs-weak solution for \eqref{eq0_st}, which consists in defining solutions to \eqref{eq0_st}
as limit of solutions to the approximated problem \eqref{eq0'_st}.
In this way, the limiting noise coefficient $G$ is allowed to be in $\cL^2(H,H)\setminus \cL^2(H,V)$, 
whereas the approximations $(G_\lambda)_\lambda$ may more regular
(e.g.,~even in $\cL^2(H,V)$).
The drawback of such weak formulation is that for the limit problem \eqref{eq0_st}
one can identify the nonlinearity $B(u)$ weakly in $V^*$, while the
identification of the nonlinearity $A(\partial_t u^d)$ (or equivalently of $\C(u)$)
is relaxed. For technical details, we refer to Definition~\ref{def:weak}.

The introduction of such relaxed notion of solution allows also to
treat the operator
$\C$ in the Kolmogorov equation. 
The  main idea is to consider a hybrid Kolmogorov equation of  the form
\begin{align}
  \nonumber
  &\alpha\varphi_\lambda(x) - \frac12\operatorname{Tr}\left[G G^*
  D^2\varphi_\lambda(x)\right]
  +\left( B_\lambda(x), D\varphi_\lambda(x)\right)_H \\
  \label{eq0'_kolm}
  &\qquad= g(x) +\left(F(x) + \C_\lambda(x), D\varphi_\lambda(x)\right)_H\,,
  \quad x\in H\,,
\end{align}
where $\C_\lambda:H\to H$ is defined as
\[
  \C_\lambda(z):= A^{-1}(F(z)-B_\lambda(z)) - (F(z)-B_\lambda(z))\,, \quad z\in H\,.
\]
The term {\it hybrid} refers to the fact that this is not the natural Kolmogorov equation 
associated to \eqref{eq0'_st}, due to the presence of the limit coefficient $G$.
Since $G$ can be taken in $\cL^2(H,H)$ and $\C_\lambda$ is no longer 
singular, existence and uniqueness of strong solutions to \eqref{eq0'_kolm}
can be obtained via fixed point arguments by exploiting strong Feller properties 
of the associated transition semigroup. Furthermore, by exploiting structural 
assumptions on the data $A$ and $g$, we show that the family 
$(\varphi_\lambda)_\lambda$ is uniformly bounded in $C^1_b(H)$.
By exploiting suitable infinite dimensional compactness arguments, 
this allows to pass to the limit in the It\^o formula for $\varphi_\lambda(u_\lambda)$,
hence to show uniqueness of the law of the possible limiting points of the 
sequence $(u_\lambda)_\lambda$.
Eventually, such uniqueness-in-law result for Friedrichs-weak solutions to \eqref{eq0_st}
can be reformulated as a weak stability result for approximations of the stochastic equation 
\eqref{eq0_st} (see Theorem~\ref{thm3} below).

We briefly summarise here the contents of the paper.
Section~\ref{sec:main} presents the mathematical setting and the
main results.
In Section~\ref{sec:non_uniq}, we present an example of
deterministic doubly nonlinear PDE with multiple solutions,
whose stochastic 
counterpart falls within our setting.
In Section~\ref{sec:exist} we prove existence of Friedrichs-weak solutions for the equation 
\eqref{eq0_st}. Section~\ref{sec:kolm} contains the analysis of the Kolmogorov equation,
while in Section~\ref{sec:uniq}, we prove the main results on uniqueness and stability by noise.


\section{Main results}
\label{sec:main}

\subsection{Notation and setting}
\label{ssec:setting}
For a given Banach space $E$, 
the symbols $\mathcal B(E)$ and $\cP(E)$ denote the 
Borel $\sigma$-algebra of $E$ and the space
of probability measures on $\mathcal B(E)$, respectively.
We indicate the space of bounded measurable real functions 
on $E$ by  $\cB_b(E)$.
Moreover, for $k\in\enne$
and $s\in(0,1]$,
the symbol $C^{k,s}_b(E)$ denotes the space of real-valued 
bounded Borel-measurable functions on $E$ which are $k$-times Fr\'echet-differentiable 
with $s$-H\"older-continuous derivatives. In particular, we
indicate by $C^{0,1}_b(E;E)$ the space of Lipschitz continuous
functions from $E$ to itself endowed with the norm
$$\|F\|_{C^{0,1}_b(E;E)}:=\sup_{v\in E}\|F(v)\|_E + \sup_{u\not = v
  \in E}\frac{\|F(u)-F(v)\|_E}{\|u-v\|_E}.$$ 

Given two Hilbert spaces $K_1$ and $K_2$, we use the symbols 
$\cL(K_1,K_2)$, $\cL^1(K_1,K_2)$, and $\cL^2(K_1,K_2)$ to
indicate the spaces of
bounded continuous linear operators, trace-class operators, and Hilbert--Schmidt operators
from $K_1$ to $K_2$, respectively. We shall use the subscript '$+$'
to indicate the respective subspaces
of nonnegative operators.

For every probability space $(\Omega,\cF,\P)$ endowed 
with a saturated and right-continuous filtration $(\cF_t)_{t\geq0}$
(i.e.~$(\cF_t)_{t\geq0}$ satisfies the usual conditions),
the progressive sigma algebra on $\Omega\times(0,+\infty)$
is denoted by $\mathcal P$.
Given a final reference time $T>0$,
we use the classical symbols $L^s(\Omega; E)$ and $L^r(0,T; E)$
for the spaces of strongly measurable Bochner-integrable functions 
on $\Omega$ and $(0,T)$, respectively, for all
$s,r\in[1,+\infty]$ and for every Banach space $E$.
If $s,r\in[1,+\infty)$ we use 
$L^s_\mP(\Omega;L^r(0,T; E))$ to indicate
that measurability is intended with respect to $\mP$.
In the case that $s\in(1,+\infty)$, $r=+\infty$, and $E$ is
separable,
we set 
\begin{align*}
  &L^s_w(\Omega; L^\infty(0,T; E^*)):=
  \big\{v:\Omega\to L^\infty(0,T; E^*) \text{ weakly* measurable: } \ 
  \E\norm{v}_{L^\infty(0,T; E^*)}^s<\infty
  \big\}\,,
\end{align*}
and recall that by 
\cite[Thm.~8.20.3]{edwards} we have the identification
\[
L^s_w(\Omega; L^\infty(0,T; E^*))=
\left(L^{\frac{s}{s-1}}(\Omega; L^1(0,T; E))\right)^*\,.
\]

Throughout the paper, $H$ is a fixed real separable Hilbert space,
with scalar product and norm denoted by $(\cdot, \cdot)_H$ and $\|\cdot\|_H$, respectively. 
The following assumptions will be assumed throughout the paper.
\begin{description}
  \item[A1] $A:H\to2^H$ is a maximal monotone operator, and 
  there exist constants $c_A,C_A>0$ such that, for all $u\in H$ and $v\in A(u)$,
  \begin{align*}
    \norm{v}_H\leq C_A(1+\norm{u}_H)\,, \qquad
    (v,u)_H\geq c_A\norm{u}_H^2 - C_A\,.
  \end{align*}
  Moreover, $A^{-1}\in C^{0,s_A}(H;H)$ for some $s_A\in(0,1)$, and 
  there is a constant $k_A>0$ such that the operator $A^{-1}-k_AI$
  ($I$ is the identity operator) has a bounded range, i.e., 
  \[
  \sup_{v\in H}\norm{A^{-1}(v) - k_A v}_H=:C_A'<+\infty\,,
\]
 for some $C_A'>0$.
  \item[A2] $\L$ is a linear, symmetric, maximal monotone, unbounded operator on $H$
  with effective domain $D(\L)$, such that $D(\L)\embed H$ compactly.
  For every $\sigma \in (0,1)$ we classically define
  $\L^\sigma$ by   spectral theory,
  we set 
  \[
    V_{2\sigma}:=D(\L^\sigma)\,, 
    \quad V:=V_1=D(\L^{1/2})\,,
  \]
  and we assume, with no loss of generality, that $(\L u,u)_H = \| u
  \|_V^2$ for all $u \in V_2 = D(\L)$.
  Moreover, $f:H\to H$ is maximal monotone and bounded, 
  i.e.,~there exists a constant $C_f>0$ such that, for all $u\in H$,
  \[
  \norm{f(u)}_H\leq C_f\,.
  \]
  The nonlinear operator $B$ (of semilinear type) is then defined as
  \[
  B:=\L + f\,,
  \]
  which is maximal monotone on $H$ with effective domain $D(B)=D(\L)=V_2$.
  \item[A3] $F:H\to H$ is
   bounded Lipschitz-continuous, and we set $C_F:=\|F\|_{C^{0,1}_b(H;H)}$.
  \item[A4] $G\in\cL^2(H,H)$ commutes with $\L$ and $\ker G =\{0\}$.
  \item[A5] $u_0\in H$.
\end{description}

We are interested in the doubly nonlinear problem 
\begin{align}
  \label{eq1}
  &\d u =(\partial_t u^d)\,\d t + G\,\d W\,,\\
  \label{eq2}
  &A(\partial_t u^d) + B(u) = F(u)\,,\\
  \label{eq3}
  &u(0)=u_0\,.
\end{align}

We consider two notions of solution for \eqref{eq1}--\eqref{eq3}, both
in the probabilistically weak sense: analytically strong solution
and analytically weak solution in the sense of Friedrichs.

\begin{defin}[Analytically strong solution]
\label{def:strong}
If $u_0\in V$ and $G\in \cL^2(H,V)$,
an \emph{analytically strong solution} to \eqref{eq1}--\eqref{eq3} 
is a family $(\Omega, \cF, (\cF_t)_{t\geq0}, \P, W, u, u^d, v)$
where, for all $T>0$,
\begin{itemize}
  \item $(\Omega, \cF, (\cF_t)_{t\geq0}, \P)$ is a filtered probability space satisfying the usual conditions,
  \item $W$ is a cylindrical Wiener process on $H$,
  \item $u\in L^2_\mP(\Omega; C^0([0,T]; H))\cap 
  L^2_w(\Omega; L^\infty(0,T; V))\cap 
  L^2_\mP(\Omega;L^2(0,T; V_2))$,
  \item $u^d \in L^2_\mP(\Omega; H^1(0,T; H))$,
  \item $v\in L^2_\mP(\Omega; L^2(0,T; H))$
\end{itemize}
and it holds that 
\begin{align}
  \label{str1}
  u(t)=u_0 + \int_0^t\partial_t u^d(s)\,\d s + \int_0^t G\,\d W(s) \qquad&\text{in } H\,, \quad
  \forall\,t\geq0\,,\quad\P\text{-a.s.}\,,\\
  \label{str2}
  v + B(u) = F(u) \qquad&\text{in } H\,, \quad\text{a.e.~in } \Omega\times(0,+\infty)\,,\\
  \label{str3}
  v \in A(\partial_t u^d) \qquad&\text{a.e.~in } \Omega\times(0,+\infty)\,.
\end{align}
\end{defin}

Existence of analytically strong solutions for doubly nonlinear equations 
\eqref{eq1}--\eqref{eq3} in the sense of Definition~\ref{def:strong}
has been proved in \cite{SS-doubly2}.
In particular, note that 
existence of analytically strong solutions cannot 
be expected solely under {\bf A4}, as the condition $G\in \cL^2(H,V)$ is needed.
The motivates to introduce a weaker notion of solution for \eqref{eq1}--\eqref{eq3}
allowing for more general noise coefficients $G$
satisfying only assumption {\bf A4}.

Before moving to the definition of the weak solution in
the sense of Friedrichs, we need an equivalent 
formulation of the concept of strong solution given by Definition~\ref{def:strong}.
First of all, observe that equation \eqref{str2} and inclusion
\eqref{str3} satisfied by analytically strong solutions 
can be written equivalently as
\[
  \partial_t u^d = A^{-1}(F(u)-B(u)) \quad\text{in }
  H\,,\quad\text{a.e.~in } \Omega\times(0, + \infty )\,.
\]
Consequently, when intended in the strong analytical sense
of Definition~\ref{def:strong}, 
problem \eqref{eq1}--\eqref{eq3}
can also be formulated as
\[
  \d u = A^{-1}(F(u)-B(u))\,\d t + G\,\d W\,, \qquad u(0)=u_0\,.
\]
Now, taking assumption {\bf A1}  into account, let us  introduce the operator 
\beq
  \label{def:C}
  \C:D(\L)\to H\,, \qquad \C(x):= A^{-1}(F(x)-B(x)) - k_A(F(x)-B(x))\,, \quad x\in V_2\,, 
\eeq
which has  bounded range, namely, 
\beq
  \label{bound_C}
  \norm{\C(x)}_H \leq C_A' \qquad\forall\,x\in V_2\,.
\eeq
Taking these remarks into account, 
an equivalent formulation of problem \eqref{eq1}--\eqref{eq3} in the strong analytical sense is
\beq
  \label{equiv}
  \d u +k_AB(u)\,\d t= \left[k_AF(u)+  \C(u)\right]\,\d t + G\,\d W\,, \qquad u(0)=u_0\,.
\eeq
\begin{defin}[Analytically strong solution: equivalent formulation]
\label{def:strong_equiv}
If $u_0\in V$ and $G\in \cL^2(H,V)$, an analytically strong solution to \eqref{equiv}
is a family $(\Omega, \cF, (\cF_t)_{t\in[0,T]}, \P, W, u)$
where, for all $T>0$,
\begin{itemize}
  \item $(\Omega, \cF, (\cF_t)_{t\geq0}, \P)$ is a filtered probability space satisfying the usual conditions,
  \item $W$ is a cylindrical Wiener process on $H$,
  \item $u\in L^2_\mP(\Omega; C^0([0,T]; H))\cap 
  L^2_w(\Omega; L^\infty(0,T; V))\cap 
  L^2_\mP(\Omega;L^2(0,T; V_2))$,
\end{itemize}
and it holds that, for every $t\geq0$, $\P$-almost surely,
\begin{align}
  \label{str_equiv}
  u(t)+k_A\int_0^tB(u(s))\,\d s=
  u_0 + \int_0^t\left[k_AF(u(s))+  \C(u(s))\right]\,\d s + \int_0^t G\,\d W(s) \quad&\text{in } H\,.
\end{align}
\end{defin}
The equivalence between Definitions~\ref{def:strong} and \ref{def:strong_equiv}
is immediate by means of position \eqref{def:C}  above.

The  doubly nonlinear evolution \eqref{eq1}--\eqref{eq3} 
inherits a semilinear structure from the structural form of $A$ and
$B$. On the other hand, 
\eqref{equiv} cannot be viewed as a classical  
semilinear stochastic equation. Indeed, the nonlinear perturbation 
$\C$ has the same differential order of the linear part $\L$ (i.e.,~$\C$ is defined on $D(\L)$),
while in classical semilinear problems the linear component 
needs to be dominant with respect to the others.

Bearing these considerations in mind, we introduce a weaker notion
of solution to \eqref{equiv}.
To this end, for every $\lambda>0$ let $B_\lambda:H\to H$ denote the Yosida approximation of $B$:
it is well known that $B_\lambda$ is $\lambda^{-1}$-Lipschitz continuous 
(see for example \cite{barbu-monot} for classical results on monotone analysis).
We define the operator 
\beq
  \label{C_lam}
  \C_\lambda:H\to H\,, \qquad 
  \C_\lambda(x):=A^{-1}(F(x)-B_\lambda(x)) - k_A(F(x)-B_\lambda(x))\,, \quad x\in H\,.
\eeq
Note that the Lipschitz-continuity of $B_\lambda$ and {\bf A1} ensure that 
$\C_\lambda\in C^{0,s_A}_b(H;H)$ with 
\beq
  \label{bound_C_lam}
  \norm{\C_\lambda(x)}_H\leq C_A' \qquad\forall\,x\in H\,.
\eeq
Definition~\ref{def:strong_equiv} can be easily adapted to the
case when $\C$ is replaced by $\C_\lambda$, so that one can speak of
analytically strong solutions for $\lambda >0$, as well. 

\begin{defin}[Friedrichs-weak solution]
\label{def:weak}
A \emph{Friedrichs-weak solution} to \eqref{eq1}--\eqref{eq3}
is a family $(\Omega, \cF, (\cF_t)_{t\geq0}, \P, W, u, y,\Lambda)$,
where, for all $T>0$,
\begin{itemize}
  \item $(\Omega, \cF, (\cF_t)_{t\geq0}, \P)$ is a filtered probability space satisfying the usual conditions,
  \item $W$ is a cylindrical Wiener process on $H$,
  \item $u\in L^2_\mP(\Omega; C^0([0,T]; H))\cap L^2_\mP(\Omega;L^2(0,T; V))$,
  \item $y \in L^\infty_\mP(\Omega\times(0,T); H)$,
  \item $\Lambda=(\lambda_n)_n\subset(0,+\infty)$,
\end{itemize}
such that there exist a sequence of data $(u_0^n, G_n)_n\subset H\times \cL^2(H,H)$,
and a sequence of analytically strong solutions
$(\Omega, \cF, (\cF_t)_{t\geq0}, \P, W, u_n)$
to the problem 
\beq
  \label{eq:n}
  \d u_n +k_AB_{\lambda_n}(u_n)\,\d t =  
  \left[k_AF(u_n)+  \C_{\lambda_n}(u_n)\right]\,\d t + G_n\,\d W\,, \qquad u_n(0)=u_0^n\,,
\eeq
satisfying, as $n\to\infty$, for all $T>0$,
\begin{align*}
  &u_0^n\to u_0 \quad\text{in } H\,,\\
  &G_n\to G \quad\text{in } \cL^2(H,H)\,,\\
  &\lambda_n\searrow 0 \quad\text{in } \erre\,,\\
  &u_n\wstarto u \quad\text{in } L^2_w(\Omega; L^\infty(0,T; H)) \cap L^2_\mP(\Omega; L^2(0,T; V))\,,\\
  &u_n\to u \quad\text{in } L^2(0,T; H)\,, \quad\P\text{-a.s.}\,,\\
  &\C_{\lambda_n}(u_n) \wstarto y \quad\text{in } L^\infty_\mP(\Omega\times(0,T); H)\,.
\end{align*}
\end{defin}

\begin{remark}\rm 
  The concept of Friedrichs-weak solution in Definition~\ref{def:weak}
  is related to the concept of closure of analytically strong solutions.
  Let us stress that in Definition~\ref{def:weak}
  the initial datum $u_0$ and the operator $G$ are allowed to satisfy 
  assumptions {\bf A4}--{\bf A5} only. In particular, neither $u_0$ 
  is required to be in $V$
  nor $G$ is required to be in $\cL^2(H,V)$. On the other hand, 
  one has much more freedom on the approximating sequences.
  These can indeed be taken either more regular 
  (e.g.~in $V$ and $\cL^2(H,V)$), or 
  simply set as $G_n=G$ for every $n$
  if one is interested 
  in the asymptotic behaviour 
  of analytically weak solutions instead.  
  Moreover, note that even if $(u_0^n)_n$ and $(G_n)_n$ are only in 
  $H$ and $\cL^2(H,H)$, respectively, the existence of analytically strong solution 
  for $n>0$ is guaranteed since $D(B_{\lambda_n})=H$ (see \cite{SS-doubly2 }).
\end{remark}

\begin{remark}\rm 
  Let us point out that if $(\Omega, \cF, (\cF_t)_{t\geq0}, \P, W, u, y,\Lambda)$
  is a Friedrichs-weak solution in the sense of Definition~\ref{def:weak} above,
  it  holds that 
  \begin{align}
  \label{we1}
  u(t)+k_A\int_0^tB(u(s))\,\d s=
  u_0 + 
  \int_0^t\left[k_AF(u(s))+ y(s)\right]\,\d s + \int_0^tG\,\d W(s) \quad\text{in } V^*
  \end{align}
  for every $t\geq0$, $\P$-almost surely.
  It is important to note however that one {\em cannot} infer that 
  $y=\C(u)$,
  due to the lack of regularity of the solution $u$ and the fact that
  $\C$ is defined on $D(\L)$. This weak notion of solution still
  features the identification of the nonlinearity $B$. In constrast,
  the limit of $\C$ is not identified and 
  the condition $y=\C(u)$ 
  needs to be interpreted in a very weak sense
  via the limiting procedure highlighted above. Equivalently, 
  this corresponds to relaxing the identification of 
  the nonlinearity $A$.
\end{remark}

Our first result is an existence proof for
Friedrichs-weak solutions.

\begin{thm}[Existence of Friedrichs-weak solutions]
  \label{thm1}
  Assume {\bf A1--A5}. Then,
   problem \eqref{eq1}--\eqref{eq3} admits at least a
  Friedrichs-weak solution in the sense of
  \emph{Definition~\ref{def:weak}} satisfying
  $u_0^n\in V$ and $G_n\in \cL^2(H,V)$ for all $n$.
\end{thm}

Our main result concerns the uniqueness in
distribution for Friedrichs-weak solutions. For this purpose, 
we need the following additional structural assumption.

\begin{description}
\item[A6] there exists $\delta\in(0,\frac12)$ such that 
  $G(H)=V_{2\delta}$
 and 
  \[
  s_A  + \frac2{1+2\delta} >2\,.
  \] 
\end{description}

  Note that the condition $\delta<\frac12$ in {\bf A6}
  implies that the range of $G$ cannot be $V$, and one has 
  $V\embed G(H)$ compactly, hence $G\notin\cL^2(H,V)$.
  On the other hand,
  the requirement $\delta>0$ is redundant 
  as it is trivially implied by the Hilbert--Schmidt conditions on $G$.
  All in all, it holds that $V\embed G(H)\embed H$ compactly.
  Moreover, we stress that for all 
  $\delta\in(0,\frac12)$, assumption {\bf A6}
  is satisfied with $s_A\in(\frac{4\delta}{1+2\delta},1)$. We have
  the following.

\begin{thm}[Uniqueness in law]
  \label{thm2}
  Assume {\bf A1--A6}, and let
  \[
  (\Omega, \cF, (\cF_t)_{t\geq0}, \P, W, u_1, y_1,\Lambda_1) \qquad\text{and}\qquad
  (\Omega, \cF, (\cF_t)_{t\geq0}, \P, W, u_2, y_2, \Lambda_2)
  \]
  be two Friedrichs-weak solutions to problem \eqref{eq1}--\eqref{eq3}
  with respect to the same initial datum $u_0$
  and defined on the same stochastic basis.
  If $\Lambda_1=\Lambda_2$, then $u_1$ and $u_2$ have the same law on 
  $C^0([0,+\infty); H)$, i.e.
  \[
  \P\left(u_1 \in E\right)=\P\left(u_2\in E\right) \quad\forall\,E\in\mathcal B(C^0([0,+\infty); H))\,.
  \]
\end{thm}

As a corollary of Theorem~\ref{thm2} we are able to establish 
a stability-by-noise result.
Indeed, 
let $\Lambda=(\lambda_n)_n$ be an infinitesimal sequence of positive real numbers, and let
$(u_0^n, G_n)_n\subset H\times \cL^2(H,H)$ such that $u_0^n\to u_0$ in $H$ and
$G_n\to G$ in $\cL^2(H,H)$.
Given the sequences $(B_{\lambda_n}, \C_{\lambda_n})_n$ of approximated operators,
problem \eqref{eq:n}  cannot be assumed to show the 
uniqueness in law property,
even if the approximating operators $(\C_{\lambda_n})_n$ are H\"older
continuous on $H$. The reason for this is
that $G_n$ may be in $\cL^2(H,V)$
(hence the Ornstein--Uhlenbeck semigroup associated to
$\L$ and $G_n$ may not be {\it strong Feller}).
Let us set then 
\[
  \mathcal S_n:=\left\{\text{analytically strong solutions of \eqref{eq:n}
  in the sense of Def.~\ref{def:strong_equiv}}\right\}
\]
and let $\cP_n$ be their respective distributions on $C^0([0,+\infty); H)$, namely
\begin{align*}
  \cP_n:=&\left\{\mu\in\cP(C^0([0,+\infty); H)):\;\exists\,
  (\Omega, \cF, (\cF_t)_{t\geq0}, \P, W, u)\in\mathcal S_n:\right.\\
  &\quad\left.
  \mu(E)=\P(u\in E) \quad\forall\,E\in \mathcal B(C^0([0,+\infty); H))\right\}\,.
\end{align*}
This means that for every $n\in\enne$, $\cP_n$ represents the set of 
distributions on $C^0([0,+\infty); H)$ of all the analytically strong solutions of \eqref{eq:n}.
The weak stability-by-noise result is the following.
\begin{thm}[Weak stability]
  \label{thm3}
  Assume {\bf A1--A6} and let
  \[
  (\lambda_n, u_0^n, G_n)_n\subset (0,+\infty)\times H\times \cL^2(H,H)
  \]
  be such that, as $n\to\infty$, $\lambda_n\searrow0$, $u_0^n\to u_0$ in $H$, and
 $G_n\to G$ in $\cL^2(H,H)$.
 Then, there exists a unique probability $\mu\in\cP(L^2_{loc}(0,+\infty; H))$
 such that for every sequence $(\mu_n)_n\subset\cP(C^0([0,+\infty); H))$ satisfying 
 \[
 \mu_n\in\cP_n \quad\forall\,n\in\enne\,,
 \]
  as $n\to\infty$, it holds that $\mu_n\wstarto\mu$ in $L^2_{loc}(0,+\infty; H)$, i.e.,
 \[
   \lim_{n\to\infty}\int_{L^2_{loc}(0,+\infty; H)}\varphi\,\d\mu_n 
   = \int_{L^2_{loc}(0,+\infty; H)}\varphi\,\d\mu
   \qquad\forall\,\varphi\in C^0_b\left(L^2(0,+\infty; H)\right)\,.
 \]
 Moreover, $\mu$ is concentrated on $C^0([0,+\infty); H)\cap L^2_{loc}(0,+\infty; V)$.
\end{thm}


\section{An explicit example}
\label{sec:non_uniq}

\subsection{Nonuniqueness for a deterministic doubly nonlinear equation}
The uniqueness-in-law result of Theorem \ref{thm2} is of a genuinely
stochastic nature and has no deterministic counterpart. In fact,
deterministic doubly nonlinear evolution equations may fail to have
unique solutions \cite{Colli,Colli-Visintin}. In order to
illustrate this fact, we follow {\sc Akagi}\cite{Akagi10} and present in this section the concrete example
of the deterministic parabolic PDE
\begin{equation}
  \label{eq:akagi0}
  |\partial_t u|^{p-2}\partial_t u -\Delta u =\lambda u \quad
  \text{in} \ (0,T)	\times D\,.
\end{equation}
Here, $D \subset {\mathbb R}^d$ $(d\in \Nz)$ is a nonempty, open, and bounded
domain with smooth boundary $\partial D$ and $p>2$. 
Note that equation \eqref{eq:akagi0} features a single
nonlinearity acting on the time derivative, which 
does not satisfy our assumption {\bf A1} in general: this issue
is
overcome at the end of the section where a modified version of 
\eqref{eq:akagi0} is  presented.
We complement equation \eqref{eq:akagi0} by
homogeneous Dirichlet boundary conditions and by a null
initial condition,
namely,
\begin{align}
  &u=0 \quad \text{on} \  (0,T)\times\partial D\,, \label{eq:akagi1}\\
  &u(\cdot,0)=0 \quad\text{on} \ D \,. \label{eq:akagi2}
\end{align}
Assuming that $\lambda > \lambda_1$ where  $\lambda_1>0$ is the first eigenvalue of the Laplacian
on $D$ with homogeneous Dirichlet conditions,
Theorem 2.1 in \cite{Akagi10}
proves that equation
\eqref{eq:akagi0} together with conditions
\eqref{eq:akagi1}--\eqref{eq:akagi2} has infinitely many strong
solutions.

Indeed, one looks for
solutions to problem \eqref{eq:akagi0}--\eqref{eq:akagi2}
having the form $u(t,x) = \theta(t)v(x)$. Here, one lets $v=0$ on
$\partial D$ and $\theta (0)=0$, so that conditions
\eqref{eq:akagi1} and \eqref{eq:akagi2} are respectively satisfied.
Equation 
\eqref{eq:akagi0} then reads
$$|\dot \theta|^{p-2}\dot \theta |v|^{p-2}v -\theta \Delta v = \theta \lambda
v \quad \text{in} \  (0,T)\times D\,.$$
One finds a solution to the latter by splitting the variables and
solving the decoupled system
\begin{alignat}{2}
  \dot \theta&=|\theta|^{p'-2}\theta \qquad&&\text{in} \ (0,T)\,,\label{eq:akagi3}\\
  -\Delta v + |v|^{p-2}v &=\lambda v \qquad&&\text{in} \ D\,, \label{eq:akagi4}
\end{alignat}
where $1/p'+1/p=1$. For all $t_*\in [0,T]$, the functions $\theta_*(t)=c_p((t-t_*)^+)^{1/(2-p')}$ with $c_p=(2-p')^{1/(2-p')}>0$
solve \eqref{eq:akagi3} with $\theta_*(0)=0$. For
later purposes, we remark that for all such functions $\theta_*$ one has
$\max_{[0,T]}|\theta_*| \leq c_p T^{1/(2-p')}=:M_1.$

As far as equation \eqref{eq:akagi4} is
concerned, one considers the minimization of the functional
$$I(v) = \int_D \left(\frac12 |\nabla v|^2 +\frac1p |v|^p -\frac{\lambda}{2}
| v|^2 \right) {\rm d} x$$
on $H^1_0(D)$. As $p>2$, global minimizers $v^*$ can be
proved to exists and  to be strong
solutions to  \eqref{eq:akagi4} with homogeneous Dirichlet boundary conditions. On the other
hand, setting $v_1\not \equiv 0$ to be the eigenfunction
related to the first eigenvalue $\lambda_1$ of the Laplacian  with
homogeneous Dirichlet boundary conditions and using the fact that
$\| \nabla v_1\|^2_{L^2(D;{\mathbb R}^d)} = \lambda_1 \|  v_1\|^2_{L^2(D )} $,  for all $s>0$ one has that
\begin{align*}
  I(sv_1) = \int_D \left(s^2\frac{\lambda_1-\lambda}{2}|v_1|^2  
+\frac{s^p}{p} |v_1|^p \right) {\rm d} x\,.
\end{align*}
As $\lambda_1-\lambda<0$, one has
that $I(sv_1)<0$ for $s$ small enough. In particular, $I(v^*)<0$ on
the global minimizer,  which implies that $v_*\not =0$. This proves that \eqref{eq:akagi4}
  has at least two distinct nontrivial solutions, namely,
$v_*$ and $-v_*$. All in all, we have checked that
\eqref{eq:akagi0}--\eqref{eq:akagi2} has at least the solutions
\begin{equation}
    u_*(t,x) = \pm v_*(x)
(2-p')^{1/(2-p')}((t-t_*)^+)^{1/(2-p')}\quad \forall (x,t)\in \Omega
\times [0,T].\label{eq:akagi5}
\end{equation}
for any given $ t_*\in [0,T]$. Note that $u_*\equiv 0$ for $t_*=T$.

Before moving on, in order to provide a nonuniqueness deterministic
counterexample which still fits to our assumptions,
let us remark that any solution $v_*\in H^1_0(D)$ of
\eqref{eq:akagi4} is bounded in $C^0(\overline{D})$ in terms of
data. Given any $q>1$, multiply equation \eqref{eq:akagi4} by $|v_*|^{q-2}v_*$
and integrate in space and by parts. By using the fact that 
$$-\int_D \Delta v_* \, |v_*|^{q-2}v_* \, {\rm d} x = (q-1)
\int_D |v_*|^{q-2}|\nabla v_*|^2\, {\rm d} x \geq 0$$
one obtains
$$\int_D |v_*|^{p+q-2}\, {\rm d} x \leq \lambda \int_D
|v_*|^{q}\, {\rm d} x ,$$
which implies that $v_*$ is bounded in $L^{p+q-2}(D)$ in terms of
data. This can indeed be made rigorous by a bootstrap
argument on $q$. In particular, one has that $v_*$ is bounded in $L^r(D)$ in terms of data for any $r>1$. Choose now $r>d(p-1)/2$. 
By a comparison in  
\eqref{eq:akagi4} one has that $-\Delta v_*$ is
bounded in $L^{r/(p-1)}(D)$ in terms of data, as well. By elliptic
regularity and the fact that $r/(p-1)>d/2$ one gets that $v_*$ is bounded in
$C^0(\overline{D})$ in terms of data. More precisely, there
exists $M_2=M_2(D,p,\lambda)$ such that $\|v\|_{C^0(\overline{D})}
\leq M_2$.

At this point, we are ready to present a deterministic counterexample 
to uniqueness, which is fitting within our setting.
Consider  the equation
\begin{equation}
  \label{eq:akagi6}
  \alpha(\partial_t u ) - \Delta u =\ell(u) \quad
  \text{in} \  (0,T)\times D\,,
\end{equation}
where the linearly bounded monotone mapping $\alpha\in C^1({\mathbb R})$ is
defined by
\begin{equation} 
  \alpha(x):=
  \left\{
    \begin{array}{ll}
      |x|^{p-2}x\quad&\text{for} \ |x|\leq M \,,\\
      (p-1)M^{p-2}x - (p-2) M^{p-1}\displaystyle\frac{x}{|x|} \quad&\text{for} \ |x|>M\,,
    \end{array}
    \right.\label{eq:akagi7}
  \end{equation}
  and $\ell\in C^{0,1}(\Rz)$ is given by
\begin{equation} 
  \ell(x):=
  \left\{
    \begin{array}{ll}
      \lambda x\quad&\text{for} \ |x|\leq M \,,\\
      \lambda M x/|x|\quad&\text{for} \ |x|>M\,,
    \end{array}
    \right.\label{eq:akagi8}
  \end{equation}
for $M:=M_1M_2$.
The same argument leading to \eqref{eq:akagi5} ensures that equation
\eqref{eq:akagi6} with conditions \eqref{eq:akagi1}--\eqref{eq:akagi2}
  has infinite many solutions. In fact, for all $u_*$ from
  \eqref{eq:akagi5} one has $|u_*|=|\theta_*||v_*| \leq M_1M_2=M$, so that
  $\alpha(\partial_t u_*) = |\partial_t u_* |^{p-2}\partial_t u_*$ and
  $\ell(u_*)= \lambda u_*$. We hence have that 
  equations \eqref{eq:akagi0} and \eqref{eq:akagi6} coincide and $u_*$
  solves \eqref{eq:akagi6}, as well. Note that the truncations allow the operators 
  $\alpha$ and $\ell$ to satisfy {\bf A1}--{\bf A3}.
  
\subsection{Restoring uniqueness by noise}
Before closing this section, let us consider the stochastic counterpart 
of the deterministic equation \eqref{eq:akagi6}. In the stochastic case, equation \eqref{eq:akagi6}
together with the boundary and initial conditions
\eqref{eq:akagi1}--\eqref{eq:akagi2} can be 
variationally formulated as \eqref{eq1}--\eqref{eq3} by
letting $H=L^2(D)$, $A(v)(x) = \alpha(v(x))$ where $\alpha$ is
given in \eqref{eq:akagi7}, $B= -\Delta$ with 
$D(B)=V_2=H^2(D)\cap H^1_0(D)$, $f=0$, and $F(v)=\ell(v)$ where
$\ell$ is given in \eqref{eq:akagi8}. By letting
$k_A:=M^{2-p}/(p-1)$ we have 
$$A^{-1}(v)(x) =
\left\{
  \begin{array}{ll}
    |v(x)|^{p'-2}v(x)&\quad \text{if} \ |v(x)|\leq M^{p-1}\,,\\[2mm]
                       k_A \left(v(x) +(p-2) M^{p-1}
    \displaystyle \frac{v(x)}{|v(x)|}\right)  &\quad \text{if} \ |v(x)|> M^{p-1}\,.
\end{array}
\right.
$$
One easily checks that $A^{-1} \in C^{0,p'-1}(H;H)$
(recall that $p>2>p'>1$) and 
\begin{align*}
  &\sup_{v\in H}\| A^{-1}(v) - k_Av\|_H^2 \\
  &\quad =\sup_{v\in H} \left(\int_{\{|v|\leq
    M^{p-1}\}} ||v|^{p'-2}v - k_A v|^2 \, {\rm d} x  +
    \int_{\{|v|>M^{p-1}\}}|k_A(p-2)M^{p-1}|^2 \, {\rm d } x\right)\\
  &\quad \leq |D| (M+k_AM^{p-1})^2  +
    |D|(k_A(p-2)M^{p-1})^2  =:(C'_A)^2<+\infty.
\end{align*}
At the same time one has that $F\in C^{0,1}_b(H,H)$ with
\begin{align*}
  &\|F\|_{C^{0,1}_b(H;H)}^2\leq 2\sup_{v\in H}\int_D |\ell(v) |^2 
    {\rm d} x + 2 \sup_{u \not = v \in H}\left(\int_D |\ell(u)-
    \ell(v)|^2 {\rm d} x\right)\left(\int_D |u-v|^2 {\rm d} x\right)^{-1}\\
  &\leq 2| D| M^2 + 2\lambda^2
    =:C_F^2<+\infty\,.
\end{align*}
Letting $G$ fulfill {\bf A4}, we have that 
condition in  {\bf A6} holds for any $\delta \in (0,1/3)$ by
taking $p>2$ large enough. Given  $u_0=0$ we have proved that 
Friedrichs-weak solution to \eqref{eq1}--\eqref{eq3} exist and are
unique in law in the sense of Theorem \ref{thm2}. On the other hand,
taking $G=0$, the deterministic equation \eqref{eq:akagi6} with $u_0=0$
has infinitely many solutions.

\section{Existence of Friedrichs-weak solutions}
\label{sec:exist}
This section is devoted to the proof of Theorem~\ref{thm1}
concerning the 
existence of Friedrichs-weak solutions.

First of all, since $V\embed H$ densely, there exists a sequence $(u_0^n)_n\subset V$
such that $\|u_0^n-u_0\|_H\to 0$ as $n\to\infty$. Moreover, let us fix a sequence $\Lambda=(\lambda_n)_n$ such that $\lambda_n\searrow0$ as $n\to \infty$.
Let us construct a sequence $(G_n)_n\subset \cL^2(H,V)$ such that 
$\|G_n-G\|_{\cL^2(H,H)}\to 0$ as $n\to\infty$.
To this end, for every $x\in H$ and $n\in\enne$
we define $G_nx:=x_n\in H$ as the unique 
solution to the singular perturbation problem 
\[
  x_n + \frac1n \L x_n = G x\,.
\]
In a more compact form, this means that we define 
$G_n:=(I+\frac1n\L)^{-1}G$, so that clearly 
$G_n\in \cL^2(H,V)$ for every $n\in\enne$. 
Let us show the convergence as $n\to\infty$.
Testing the equation above by $x_n$ yields directly
\[
  \frac12\| G_nx\|_H^2 + 
  \frac1n\| G_nx\|_V^2 \leq \frac12\|Gx\|_H^2 \quad\,\forall\,x\in H\,,
\]
so that
\[
  \limsup_{n\to\infty}\|G_nx\|_H^2\leq\|Gx\|_H^2 \quad\,\forall\,x\in H\,,
\]
hence $G_nx\to Gx$ in $H$ for every $x\in H$.
Moreover, given a 
complete orthonormal system $(e_j)_j$ of $H$,
choosing $x=e_j$, and summing over $j\in\enne$, we also obtain 
\[
  \frac12\| G_n\|_{\cL^2(H,H)}^2 + 
  \frac1n\|G_n\|_{\cL^2(H,V)}^2 \leq \frac12\|G\|_{\cL^2(H,H)}^2\,.
\]
This implies analogously that $ G_n\to G$ in $\cL^2(H,H)$ as $n\to\infty$.

Since $\C_{\lambda_n}$ and $F$ are Lipschitz-continuous on $H$
for every $n\in\enne$, there exists an
analytically strong solution $(\Omega, \cF, (\cF_t)_{t\in[0,T]}, \P, W, u_n)$
in the sense of Definition~\ref{def:strong_equiv} for the problem
with $\C$ replaced by $\C_{\lambda_n}$, i.e.,
\[
  u_n\in L^2_\mP(\Omega; C^0([0,T]; H))\cap 
  L^2_w(\Omega; L^\infty(0,T; V))\cap 
  L^2_\mP(\Omega;L^2(0,T; V_2))
\]
and
\[
  u_n(t)+k_A\int_0^tB(u_n(s))\,\d s=
  u_0^n + \int_0^t\left[k_AF(u_n(s))+  
  \C_{\lambda_n}(u_n(s))\right]\,\d s + \int_0^tG_n\,\d W(s) \quad\text{in } H\,,
\]
for every $t\geq0$, $\P$-almost surely.

Now, the It\^o formula for the square of the $H$-norm yields 
\begin{align*}
  &\frac12\|u_n(t)\|_H^2 + k_A\int_0^t\|u_n(s)\|_V^2\,\d s\\
  &=\frac12\|u_0^n\|_H^2 + \int_0^t\left(k_AF(u_n(s)) + \C_{\lambda_n}(u_n(s))-
  k_Af(u_n(s)), u_n(s)\right)_H\,\d s\\
  &\qquad+\frac12\int_0^t\|G_n(s)\|^2_{\cL^2(H,H)}\,\d s
  +\int_0^t\left(u_n(s), G_n\d W(s)\right)_H\,.
\end{align*}
By {\bf A2}--{\bf A3}, \eqref{bound_C_lam}, and the properties of
$(u_0^n,G_n)_n$ we deduce that 
\begin{align*}
  \frac12\|u_n(t)\|_H^2 + k_A\int_0^t\|u_n(s)\|_V^2\,\d s
  \leq 
  C\left(1+\int_0^t\|u_n(s)\|_H^{2}\,\d s\right)
  +\int_0^t\left(u_n(s), G_n\d W(s)\right)_H
\end{align*}
for some constant $C>0$ independent of $n$.
 A standard application of the 
the Burkholder--Davis--Gundy inequality 
together with the Gronwall lemma  implies that 
\[
  \|u_n\|_{L^2_\mP(\Omega; C^0([0,T]; H)\cap L^2(0,T; V))}\leq C\,.
\]
Moreover, we recall also that by \eqref{bound_C_lam} and {\bf A2} we have
\[
  \|\C_{\lambda_n}(u_n)\|_{L^\infty_\mP(\Omega\times(0,T); H)}
  +\|f(u_n)\|_{L^\infty_\mP(\Omega\times(0,T); H)}\leq C\,,
\]
while from the convergence $G_n\to G$ in $\cL^2(H,H)$ it follows that 
\[
  \int_0^\cdot G_n\,\d W(s)\to \int_0^\cdot G\,\d W(s)
  \quad\text{in } L^2_\mP(\Omega; C^0([0,T]; H))\,.
\]
Since $C^0([0,T]; H)\cap L^2(0,T; V)$ is compactly embedded in
$L^2(0,T; H)$  as effect of the Aubin--Lions Lemma, 
by identifying $W$ with a constant
sequence of random variables in $C^0([0,T];\tilde H)$, where
$\tilde H$ is a Hilbert--Schmidt extension of $H$,
by the Prokhorov and Skorokhod theorem there exists
a probability space $(\hat \Omega, \hat\cF, \hat\P)$ 
and a sequence $(\Gamma_n)_n$ such that $\Gamma_n:\hat\Omega\to\Omega$
satisfies $(\Gamma_n)_{\#}\hat\P=\P$ for every $n\in\enne$  and
\begin{align*}
  \hat u_n:=u_n\circ\Gamma_n\to \hat u \quad&\text{in } 
  L^2(0,T; H)\,, \quad\hat\P\text{-a.s.}\,,\\
  \hat u_n\wstarto \hat u \quad&\text{in } 
  L^2_w(\hat\Omega; L^\infty(0,T;H))\cap L^2(\hat\Omega; L^2(0,T; V))\,,\\
  \C_{\lambda_n}(\hat u_n)\wstarto \hat y
  \quad&\text{in } 
  L^\infty(\hat\Omega\times(0,T);H)\,,\\
  \hat W_n:=W\circ\Gamma_n \to \hat W \quad&\text{in } C^0([0,T];\tilde H)\,,
  \quad\hat\P\text{-a.s.}\,,\\
  \hat I_n:=\left(\int_0^\cdot G_n\,\d W(s)\right)\circ\Gamma_n \to \hat I
  \quad&\text{in } C^0([0,T];H)\,,
  \quad\hat\P\text{-a.s.}\,,
\end{align*}
for some processes 
\begin{align*}
  \hat u &\in L^2_w(\hat \Omega; L^\infty(0,T; H))\cap L^2(\hat\Omega; L^2(0,T; V))\,,\\
  \hat y  &\in L^\infty(\hat\Omega\times(0,T);H)\,,\\
  \hat W &\in L^2(\hat \Omega; C^0([0,T]; \tilde H))\,,\\
  \hat I &\in L^2(\hat \Omega; C^0([0,T]; H))\,.
\end{align*}
In particular, these convergences imply also that 
\[
  \L\hat u_n\wto\L\hat u \quad\text{in } L^2(\hat\Omega; L^2(0,T; V^*))\,,
\]
and, by the strong-weak closure of maximal monotone operator $f$ 
and by  the Lipschitz-continuity of $F$,
\begin{align*}
  f(\hat u_n)\wstarto f(\hat u) \quad&\text{in } L^\infty(\hat\Omega\times(0,T);H)\,,\\
  F(\hat u_n)\wstarto F(\hat u) \quad&\text{in } L^2(0,T;H)\,,\quad\hat\P\text{-a.s.}
\end{align*}
 Define  now the filtrations $(\hat\cF_{n,t})_t$ and
$(\hat\cF_{t})_t$  by 
\[
  \hat\cF_{n,t}:=\sigma\left\{\hat u_n(s), \hat I_n(s), \hat W_n(s): s\in[0,t] \right\}\,,
  \quad
  \hat\cF_{t}:=\sigma\left\{\hat u(s), \hat y(s), \hat I(s), \hat W(s): s\in[0,t] \right\}\,.
\]
It is a standard matter (see \cite{scar-NON21} for details) 
to check that $\hat W_n$ and $\hat W$
are $H$-cylindrical Wiener processes with respect to 
$(\hat\cF_{n,t})_t$ and $(\hat\cF_{t})_t$, respectively.
Moreover, by using classical martingale arguments, it also holds that 
$\hat I_n=\int_0^\cdot G_n\,\d \hat W_n(s)$ and 
$\hat I=\int_0^\cdot G\,\d \hat W(s)$.
We now note that 
\[
  \hat u_n(t)+k_A\int_0^tB(\hat u_n(s))\,\d s=
  u_0^n + \int_0^t\left[k_AF(\hat u_n(s))+  
  \C_{\lambda_n}(\hat u_n(s))\right]\,\d s + \int_0^tG_n\,\d \hat W_n(s) \quad\text{in } H\,,
\]
for every $t\geq0$, $\hat\P$-almost surely. By using the convergences above and 
classical arguments we infer then that 
$\hat u\in L^2(\hat\Omega; C^0([0,T]; H))$ with 
\[
  \hat u(t)+k_A\int_0^tB(\hat u(s))\,\d s=
  u_0 + \int_0^t\left[k_AF(\hat u(s))+  
  \hat y(s)\right]\,\d s + \int_0^tG\,\d \hat W(s) \quad\text{in } H\,,
\]
for every $t\geq0$, $\hat\P$-almost surely. This concludes the proof of 
Theorem~\ref{thm1}.


\section{The Kolmogorov equation}
\label{sec:kolm}
 Let us focus now  on the Kolmogorov equation associated to the stochastic evolution 
equation \eqref{eq1}--\eqref{eq3}, as this will be crucial in proving uniqueness.

The equivalent reformulation  \eqref{equiv} of the doubly nonlinear problem 
\eqref{eq1}--\eqref{eq3} suggests the form of the associated Kolmogorov equation, namely 
\begin{align}
  \nonumber
  &\alpha\varphi(x) - \frac12\operatorname{Tr}\left[QD^2\varphi(x)\right]
  +k_A\left(\L x, D\varphi(x)\right)_H \\
  \label{eq:kolm}
  &\qquad= g(x) +\left(k_AF(x) + \C(x) - k_Af(x), D\varphi(x)\right)_H\,,
  \quad x\in V_2\,,
\end{align}
where $\alpha>0$ is fixed and $g$ is a suitable forcing term.
Clearly, due  to the singularity of the perturbation $\C$, in general 
one cannot expect to prove existence of mild solutions to \eqref{eq:kolm}
via classical arguments, i.e., by  directly exploiting
the  possible regularising properties of the Ornstein--Uhlenbeck
semigroup. The idea is then to characterise the solution of \eqref{eq:kolm}
is a different way, still relying on approximating arguments \`a la Friedrichs.

From now on, for brevity of notation we will assume with no loss
of generality that 
\[
k_A=1\,.
\]
Furthermore, we recall that by assumptions {\bf A2}--{\bf A3} one has
\beq
  \label{bound_F}
  \norm{F(x)}_H \leq C_F\,, \quad \norm{f(x)}_H\leq C_f \qquad\forall\,x\in H\,.
\eeq
This will be essential in the analysis of the Kolmogorov equation.

\subsection{The Ornstein--Uhlenbeck semigroup}
\label{ssec:OU}
We collect here some preliminary results on the 
Ornstein--Uhlenbeck semigroup associated to \eqref{eq:kolm}.
First, note that by assumption {\bf A2} we have that $-\L$ generates 
a strongly continuous semigroup $(e^{-t\L})_{t\geq0}$ of contractions on $H$.
Introducing then the operator 
\beq
\label{Qt}
  Q_t:=\int_0^te^{-2s\L}Q\,\d s\,, \quad t\geq0\,,
\eeq
with $ Q =G^*G$, 
it is clear that $Q_t\in \cL^1(H,H)$ for every $t\geq0$ since $G\in\cL^2(H,H)$
by assumption {\bf A4}.
With this notation, we define the Ornstein--Uhlenbeck semigroup as
\beq
  \label{Rt}
  (R_t\varphi)(x):=\int_H\varphi(e^{-t\L}x + y)N_{Q_t}(\d y)\,, \quad x\in H\,, \quad \varphi\in\cB_b(H)\,,
  \quad t\geq0\,,
\eeq
where $N_{Q_t}$ denotes a Gaussian measure on $H$, centred at $0$,
and with covariance operator $Q_t$. 
Note that an easy computation shows that 
\[
  Q_t=\frac12\L^{-1}\left(I-e^{-2t\L}\right)Q \quad\forall\,t\geq0\,.
\]

\begin{lem}[Strong Feller property]
  \label{lem:OU}
  Assume {\bf A2},  {\bf A4}, and {\bf A6}. Then, the semigroup $R$ is strong Feller. 
  More specifically, 
  there exists a constant $C_{R}>0$ such that,
  for every $t>0$ and $\varphi\in\cB_b(H)$, it holds that 
  $R_t\varphi\in C^\infty_b(H)$ and
  \begin{align}
    \label{SF1}
    \sup_{x\in H}|(R_t\varphi)(x)| &\leq \sup_{x\in H}|\varphi(x)|\,,\\
    \label{SF2}
    \sup_{x\in H}\|D(R_t\varphi)(x)\|_H &\leq 
    \frac{C_{R}}{t^{\frac12+\delta}}\sup_{x\in H}|\varphi(x)|\,.
  \end{align}
    Moreover, for every $\varphi \in C^1_b(H)$, it holds that 
  \begin{align}
    \label{SF2_bis}
    \sup_{x\in H}\|D^2(R_t\varphi)(x)\|_{\cL(H,H)} &\leq 
    \frac{C_{R}}{t^{\frac12+\delta}}\sup_{x\in H}\|D\varphi(x)\|_H\,.
  \end{align} 
  Furthermore, for every 
  $\eps\in(0,\frac12-\delta)$
 and 
  $\eta\in (0,\frac{1-2\delta}{1+2\delta})$,
   there exists constants
  $C_{R,\eps},C_{R,\eta}>0$ such that, 
  for every $t>0$ and $\varphi\in\cB_b(H)$, it holds that 
  $D(R_t\varphi)(H)\subset V_{2\eps}$ and
  \begin{align}
  \label{SF3}
  \sup_{x\in H}\|\L^\eps D(R_t\varphi)(x)\|_H &\leq 
    \frac{C_{R,\eps}}{t^{\frac12+\delta+\eps}}\sup_{x\in H}|\varphi(x)|\,,\\
  \label{SF4}
  \norm{R_t\varphi}_{C^{1,\eta}(H)} &\leq 
  \frac{C_{R,\eta}}{t^{(1+\eta)(\frac12+\delta)}}\sup_{x\in H}|\varphi(x)|
  \end{align}
\end{lem}

The proof of Lemma~\ref{lem:OU} follows immediately from
\cite[Lem.~4.1--4.2]{BOS_uniq}.

\begin{remark}\rm 
Note that since by assumption {\bf A6} one has $\delta\in(0,\frac12)$, it follows that 
the exponent $\frac12+\delta$ also belongs to $(0,1)$. Moreover,
the conditions on $\eps\in(0,\frac12-\delta)$ and $\eta\in(0,\frac{1-2\delta}{1+2\delta})$ 
ensure that also $\frac12+\delta+\eps\in(0,1)$
and $(1+\eta)(\frac12+\delta)\in (0,1)$, as one can easily check.
\end{remark}

\subsection{The regularised Kolmogorov equation}
The main  technical challenge in the study of the Kolmogorov
equation \eqref{eq:kolm} is that of taming the singularity of $\C$. We
proceed by approximation and, for all $\lambda>0$, argue with  $\C_\lambda\in C^{0,s_A}_b(H;H)$, 
which satisfies the uniform bound \eqref{bound_C_lam}, and  with
 the Yosida approximation 
$f_\lambda:H\to H$ of $f$.

Given a fixed $\alpha>0$ and a forcing $g\in C^0_b(H)$
 (some additional qualification for $\alpha$ and $g$ will be
introduced later),
the regularised Kolmogorov equation reads, for $\lambda>0$,
\begin{align}
  \nonumber
  &\alpha\varphi_{\lambda}(x)
  - \frac12\operatorname{Tr}\left[QD^2\varphi_{\lambda}(x)\right]
  +\left(\L x, D\varphi_{\lambda}(x)\right)_H \\
  \label{eq:kolm_reg}
  &\qquad= g(x) +\left(F(x) + \C_{\lambda}(x) - f_\lambda(x), D\varphi_{\lambda}(x)\right)_H\,,
  \quad x\in V_2\,.
\end{align}
Solutions to the regularised Kolmogorov equation \eqref{eq:kolm_reg} can be
intended either in the mild sense, by exploiting the representation formula for the resolvent of the semigroup $R$,
or in the strong classical sense,
as specified in the following.
\begin{defin}[Mild solution]
  \label{def:mild_kolm}
  A function $\psi\in C^1_b(H)$  is said to be a
  \emph{mild solution} to \eqref{eq:kolm_reg} if 
  \[
  \psi(x)=\int_0^{+\infty}e^{-\alpha t}R_t
  \left[g+\left(F+ \C_{\lambda} - f_\lambda, 
  D\psi\right)_H\right](x)\,\d t \qquad\forall\,x\in H\,.
  \]
\end{defin}
\begin{defin}[Classical solution]
  \label{def:class_kolm}
   A function $\psi\in C^2_b(H)$ is said to be a \emph{classical
    solution} to \eqref{eq:kolm_reg} if it is a mild solution and it satisfies,
  for every $x\in V_2$,
  \begin{align*}
  &\alpha\psi(x)
  - \frac12\operatorname{Tr}\left[QD^2\psi(x)\right]
  +\left(\L x, D\psi(x)\right)_H 
  = g(x) +\left(F(x) + \C_{\lambda}(x) - f_\lambda(x), D\psi(x)\right)_H\,.
\end{align*}
\end{defin}
The following result ensures that the regularised Kolmogorov equation
\eqref{eq:kolm_reg} is well posed 
in both the mild and strong sense, and provides estimates on the solution $\varphi_{\lambda}$
 which are uniform  with respect to the parameter $\lambda$.
\begin{prop}[Well-posedness of the Kolmogorov equation]
  \label{prop:kolm}
  Assume {\bf A1--A6}, let $g\in C^{0,s_A}_b(H)$, 
  $\eps\in(0,\frac12-\delta)$, and $\eta\in(0,\frac{1-2\delta}{1+2\delta})$.
  Then, there exists $\alpha_0>0$ such that, for every
  $\alpha>\alpha_0$, the following holds:
  \begin{itemize}
  \item[(i)] for every $\lambda>0$,
  equation \eqref{eq:kolm_reg} admits a unique mild solution 
  $\varphi_{\lambda}\in C^1_b(H)$;
  \item[(ii)] there exists a constant $ C_1>0$, independent of $\lambda$, such that 
  \beq
    \label{est1}
    \norm{\varphi_{\lambda}}_{C^1_b(H)} \leq  C_1  \qquad\forall\,\lambda>0\,;
  \eeq 
  \item[(iii)]  there exists a constant $C_\eps>0$, 
  independent of $\lambda$, such that $D\varphi_{\lambda}(H)\subset V_{2\eps}$ and
  \beq
    \label{est2}
    \norm{\L^\eps D\varphi_{\lambda}}_{C^0_b(H;H)} \leq C_\eps \qquad\forall\,\lambda>0\,;
  \eeq 
  \item[(iv)] there exists a constant $C_\eta>0$, 
  independent of $\lambda$, such that
  \beq
    \label{est3}
    \norm{\varphi_{\lambda}}_{C^{1,\eta}_b(H)} \leq C_\eta \qquad\forall\,\lambda>0\,;
  \eeq 
  \item[(v)] for every $\lambda>0$,
  the unique mild solution satisfies $\varphi_{\lambda}\in C^2_b(H)$ and
  is also a classical solution of equation \eqref{eq:kolm_reg};
  \item[(vi)] there exists a constant $ C_2>0$, independent of $\lambda$, such that
  \beq
    \label{est4}
    \norm{\varphi_{\lambda}}_{C^{2}_b(H)} \leq  C_2
    \qquad\forall\,\lambda>0\,.
  \eeq 
  \end{itemize}
\end{prop}
\begin{proof}[Proof of Proposition~\ref{prop:kolm}]
(i). 
We exploit a fixed point argument. Let 
  $ S_\alpha: C^1_b(H)\to C^0_b(H)$ be defined as
  \[
    S_\alpha v(x):=\int_0^{\infty}e^{-\alpha t} 
  R_t[g+(F+\C_\lambda-f_\lambda,Dv)_H](x)\,\d t\,,
  \quad v\in C^1_b(H)\,.
  \]
  Note that
  \begin{align*}
  &e^{-\alpha t} \|DR_t[g+(F+\C_\lambda-f_\lambda,Dv)_H](x)\|_{H}\\
  &\hspace{5.3mm}\stackrel{\eqref{SF2}}{\leq} 
  C_R\frac{e^{-\alpha t}}{t^{\frac12+\delta}}
  \sup_{y\in H}|g(y)+(F(y)+\C_\lambda(y)-f_\lambda(y),Dv(y))_H|\\
  &\qquad\leq C_R\left(\|g\|_{C^0_b(H)} + 
  (C_F+C_A'+C_f)\norm{Dv}_{C^0_b(H;H)}\right)
  \frac{e^{-\alpha t}}{t^{\frac12+\delta}}\,.
  \end{align*}
  Hence, thanks to the 
  Dominated Convergence Theorem and  to a  differentiation
  under  the   integral sign, 
  we get that $ S_\alpha$ 
  is well-defined as a map from $C^1_b(H)$ into itself, and it holds that
  \[
  D( S_\alpha v)(x)=
  \int_0^{\infty}e^{-\alpha t} DR_t[g+(F+\C_\lambda-f_\lambda,Dv)_H](x)\,\d t\,,
  \quad v\in C^1_b(H)\,.
  \]
  Furthermore, 
  for every $v_1,v_2\in C^1_b(H)$, by arguing on the difference of the 
  respective equations it is immediate to see that 
  \[
  | S_\alpha v_1(x) -  S_\alpha v_2(x)|
  \leq \frac1\alpha(C_F+C_A'+C_f)
  \norm{Dv_1-Dv_2}_{C^0_b(H;H)}
  \]
  and 
  \begin{align*}
  \|D  S_\alpha v_1(x) - D  S_\alpha v_2(x)\|_{H}&\leq
  C_R(C_F+C_A'+C_f)\norm{Dv_1-Dv_2}_{C^0_b(H;H)}
  \int_0^{+\infty}e^{-\alpha t} \frac1{t^{\frac12+\delta}}\,\d t\\
  &=\frac{C_R}{\alpha^{\frac12-\delta}}\Gamma
  \left(\frac12-\delta\right)
  (C_F+C_A'+C_f)\norm{Dv_1-Dv_2}_{C^0_b(H;H)}
  \end{align*}
   where $\Gamma$ is the Riemann $\Gamma$-function. 
  It follows that $ S_\alpha$ is a contraction on $C^1_b(H)$
  provided that $\alpha$ is chosen sufficiently large, e.g., for 
  $\alpha>\alpha_0$ for some $\alpha_0$ only depending on 
  $C_F, C_A', C_f, \delta, C_R$. Existence and uniqueness of a mild 
  solution is then a consequence of the Banach Fixed Point Theorem.\\
  (ii). This is an immediate consequence of the computations in point (i).
  Indeed, since for $\alpha>\alpha_0$ one has in particular that
  $\frac{C_R}{\alpha^{\frac12-\delta}}\Gamma
  \left(\frac12-\delta\right)
  (C_F+C_A'+C_f)<1$  and  the already performed computations yield
  \[
  \left(1- \frac{C_R}{\alpha^{\frac12-\delta}}\Gamma
  \left(\frac12-\delta\right)
  (C_F+C_A'+C_f)\right)\|D\varphi_\lambda\|_{C^0_b(H;H)} \leq 
  \frac{C_R}{\alpha^{\frac12-\delta}}\Gamma
  \left(\frac12-\delta\right)\|g\|_{C^0_b(H)}\,,
  \]
  hence \eqref{est1} follows.\\
  (iii). 
  By \eqref{SF3}
  and the Dominated Convergence Theorem,
  we have that 
  \begin{align*}
  \norm{\L^\eps D\varphi_\lambda}_{C^0_b(H;H)}&\stackrel{\eqref{SF3}}{\leq}
  \int_0^\infty e^{-\alpha t}\frac{C_{R,\eps}}{t^{\frac12+\delta+\eps}}
  \sup_{x\in H}\left(|g(x)| + 
  (C_F+C_A'+C_f)
  \norm{D\varphi_\lambda(x)}_{H}\right)\,\d t\\
  &\hspace{1.2mm} \leq \frac{C_{R,\eps}}{\alpha^{\frac12-\delta-\eps}}
  \Gamma\left(\frac12-\delta-\eps\right)
  \left(\norm{g}_{C^0_b(H)} + 
  (C_F+C_A'+C_f)
  \norm{D\varphi_\lambda}_{C^0_b(H;H)}\right)\,.
  \end{align*}
  Hence,  by possibly choosing a larger 
  $\alpha_0$ such that 
  for all $\alpha> \alpha_0$ it also holds that 
  \[
  \frac{C_{R,\eps}}{\alpha^{\frac12-\delta-\eps}}
  \Gamma\left(\frac12-\delta-\eps\right)
  (C_F+C_A'+C_f)<1\,,
  \]
  estimate \eqref{est2} follows. \\
  (iv).
  For every $\eta\in(0, (1-2\delta)/(1+2\delta))$, by \eqref{SF4} one has,
  for all $t>0$,
  \begin{align*}
  &e^{-\alpha t}\norm{DR_t[g+(F+\C_\lambda-f_\lambda,D\varphi_\lambda)_H]}_{C^{0,\eta}_b(H)}\\
  &\qquad\stackrel{\eqref{SF4}}{\leq} \frac{C_{R,\eta}}{t^{(1+\eta)(\frac12+\delta)}}e^{-\alpha t}
  \left(\norm{g}_{C^0_b(H)} + 
  (C_F+C_A'+C_f)
  \norm{D\varphi_\lambda}_{C^0_b(H;H)}\right) \,,
  \end{align*}
  where $(1+\eta)(\frac12+\delta)\in (0,1)$. 
  Using the Dominated Convergence Theorem and estimate \eqref{est1}
   one has  that $\varphi_\lambda\in C^{1,\eta}_b(H)$ and 
  \begin{align*}
    \norm{\varphi_\lambda}_{C^{1,\eta}_b(H)}\leq 
    C_{R,\eta}
    \left(\norm{g}_{C^0_b(H)} + 
     C_1(C_F+C_A'+C_f)\right)
    \int_0^\infty\frac{1}{t^{(1+\eta)(\frac12+\delta)}}e^{-\alpha t}\,\d t\,,
  \end{align*}
  so that also the estimate \eqref{est3} is proved.\\
  (v). We use the Schauder estimates as in \cite[Prop.~6.4.2]{DapZab2}.
  Let $\eta\in (0, (1-2\delta)/(1+2\delta))$ and set $\eta_0:=\min\{s_A, \eta\}$.
  Then,  recalling that $g\in C^{0,s_A}_b(H)$, 
  we note that
\[
  x\mapsto g(x) + (F(x)+\C_\lambda(x)-f_\lambda(x),D\varphi_\lambda(x))_H
\]
is $\eta_0$-H\"older continuous and bounded in $H$. It follows then
by \cite[Prop.~6.4.2]{DapZab2} 
that
\[
\varphi_\lambda
\in C^{\eta_0+\frac{2}{1+2\delta}}_b(H)\,.
\]
This implies then that $$D\varphi_\lambda\in C^{0,\eta_0+\frac{1-2\delta}{1+2\delta}}_b(H;H)\,.$$
By iterating the argument, this implies that 
\[
  x\mapsto g(x) + (F(x)+\C_\lambda(x)-f_\lambda(x),D\varphi_\lambda(x))_H
\]
is $\min\{s_A, \eta_0+\frac{1-2\delta}{1+2\delta}\}$-H\"older continuous and bounded in $H$.
Again by \cite[Prop.~6.4.2]{DapZab2} we infer that
that
\[
\varphi_\lambda
\in C^{\min\{s_A+\frac{2}{1+2\delta},
 \eta_0+\frac{3-2\delta}{1+2\delta}\}}_b(H)\,.
\]
This implies then that $$D\varphi_\lambda\in 
C^{0,\min\{s_A+\frac{1-2\delta}{1+2\delta},
 \eta_0+2\frac{1-2\delta}{1+2\delta}\}}_b(H;H)\,.$$
By iterating the argument, there exists $k$ sufficiently large 
such that $ \eta_0+2\frac{1-2\delta}{1+2\delta}>s_A$: hence, 
one deduces that 
\[
\varphi_\lambda
\in C^{s_A+\frac{2}{1+2\delta}}_b(H)\,.
\]
  Since
  \[
  s_A + \frac2{1+2\delta} >2
  \]
  by assumption {\bf A6}, 
 this implies 
in particular that $ \varphi_\lambda \in C^2_b(H)$.
Moreover, the same result \cite[Prop.~6.4.2]{DapZab2}  also
ensures   that
$\varphi_\lambda$ is a strong solution to the regularised 
Kolmogorov equation in the sense of 
Definition~\ref{def:class_kolm}.\\
(vi). Again by \cite[Prop.~6.4.2]{DapZab2}, 
since 
\[
x\mapsto g(x) + (F(x)+\C_\lambda(x)-f_\lambda(x),D\varphi_\lambda(x))_H
\]
is bounded in $C^{0,s_A}_b(H)$ uniformly in $\lambda$,
one has that 
\[
  \|\varphi_\lambda\|_{C^2_b(H)}\leq C\alpha^{2\delta-s_A(\frac12+\delta)}=:C_2
\]
for a positive constant $C$ independent of $\lambda$,
so that 
\eqref{est4} follows.
\end{proof}

\subsection{Passage to the limit in the Kolmogorov equation}
\label{ssec:lim_kolm}
We are now ready to  address  the asymptotic behaviour 
of the solutions $(\varphi_{\lambda})_{\lambda>0}$ 
to the Kolmogorov equation \eqref{eq:kolm_reg} as  $\lambda\to
0$.
We show that
the sequence of approximated solutions $(\varphi_{\lambda})_{\lambda>0}$
admits  a limit  $\varphi$, which may be interpreted
as a Friedrichs-type solution of the limiting Kolmogorov equation \eqref{eq:kolm}.
This is rigorously stated in the following result.

\begin{prop}[Convergence of $\varphi_\lambda$]
  \label{prop:lim_kolm}
  Assume {\bf A1--A6}, let $g\in C^{0,s_A}_b(H)$, let $\alpha>\alpha_0$ be fixed,
  and let $(\varphi_{\lambda})_{\lambda>0}$ be the family of solutions
  to \eqref{eq:kolm_reg} as given in \emph{Proposition~\ref{prop:kolm}}.
  Then, for every sequence $\Lambda=(\lambda_n)_n$ 
  with $\lambda_n\searrow0$ as $n\to\infty$,
  there exists a unique 
  \[
  \varphi\in  C_b^{1,\frac{1-2\delta}{1+2\delta}-}(H)  \qquad\text{with}\qquad
  D\varphi\in  C_b^0(H;D(\L^{(\frac12-\delta)-})\,,  
   \] 
  depending only on $A$, $B$, $F$, $G$, and $\Lambda$, 
  and a subsequence $(\varphi_{\lambda_{n_k}})_{k\in\enne}$ satisfying, as $k\to\infty$,
  \begin{align*}
  \lim_{k\to\infty}|\varphi_{\lambda_{n_k}}(x)-\varphi(x)|=0&\qquad\forall\,x\in H\,,\\
  \lim_{k\to\infty}\|D\varphi_{\lambda_{n_k}}(x)- D\varphi(x)\|_H=0&\qquad\forall\,x\in H\,.
  \end{align*}
\end{prop}
\begin{proof}[Proof of Proposition~\ref{prop:lim_kolm}]
{\sc Step 1.}  For every $r>0$,  set $
  B^V_r:=\{x\in V:\|x\|_V\leq r\}$
 to be 
the $V$-closed ball of radius $r$. Let $\Lambda=(\lambda_n)_n$ be an arbitrary sequence satisfying 
$\lambda_n\searrow0$ as $n\to\infty$.
 Let also $j\in\enne_+$ be fixed
and let us consider $(\varphi_{\lambda_n})_n$ restricted to $B_j^V$:
since $V\embed H$ compactly, the ball $B_j^V$ endowed with the 
$H$-distance is a compact metric space.
Thanks to \eqref{est1} it readily follows that the sequence 
$\varphi_{\lambda_n}:B_j^V\to\erre$
are uniformly equicontinuous and bounded: hence, by the Ascoli--Arzel\`a Theorem there exists 
a subsequence $(\varphi_{\lambda^j_{n_k}})_k$ and
$ \varphi_j  : B_j^V\to\erre$, continuous with respect to the topology of $H$, such that 
\[
  \lim_{k\to\infty}\sup_{x\in B_j^V}|\varphi_{\lambda^j_{n_k}}(x)-\varphi_j(x)|=0\,.
\]
Clearly, it follows that $\varphi_{j+1}=\varphi_j$ on $B_j^V$ for every $j\in\enne_+$, so that 
the family $(\varphi_j)_{j>0}$ uniquely determines a function $\varphi:V\to\erre$.
Moreover, via a standard diagonal argument 
one can work with the same subsequence $(\varphi_{\lambda_{n_k}})_k$
for every $j\in\enne$.  Specifically, we have 
\beq\label{aux_conv1}
  \lim_{k\to\infty}\sup_{x\in B_{j}^V}|\varphi_{\lambda_{n_k}}(x)-\varphi(x)|=0 \quad\forall\,j\in\enne_+\,.
\eeq
In particular, this implies the pointwise convergence 
\[
  \lim_{k\to\infty}\varphi_{\lambda_{n_k}}(x) = \varphi(x) \quad\forall\,x\in V\,.
\]
Let us prove that $\varphi$ uniquely extends to a function in $C^{0,1}_b(H)$.
To this end, thanks to \eqref{est1} we note that for every $x_1,x_2\in V$ we have 
\begin{align*}
  &|\varphi(x_1)-\varphi(x_2)|\\
  &\leq 
  |\varphi(x_1) - \varphi_{\lambda_{n_k}}(x_1)|
  +|\varphi_{\lambda_{n_k}}(x_1) - \varphi_{\lambda_{n_k}}(x_2)|
  +|\varphi_{\lambda_{n_k}}(x_2)-\varphi(x_2)|\\
  &\leq |\varphi(x_1) - \varphi_{\lambda_{n_k}}(x_1)|+ C_1\norm{x_1-x_2}_H 
  +|\varphi_{\lambda_{n_k}}(x_2)-\varphi(x_2)|
\end{align*}
for every $k\in\enne$, so that letting $k\to\infty$ yields 
\[
  |\varphi(x_1)-\varphi(x_2)|\leq  C_1\norm{x_1-x_2}_H \quad\forall\,x_1,x_2\in V\,.
\]
By density of $V$ in $H$ this shows indeed that $\varphi\in C^{0,1}_b(H)$, 
where we have used the same symbol for the extension. 
Moreover, for every $x\in H$, let $(x_m)_{m\in\enne}\subset V$ be such that 
$x_m\to x$ in $H$ as $m\to\infty$. Then, again by \eqref{est1} and the 
Lipschitz-continuity of $\varphi_{\lambda_{n_k}}$ and  $\varphi$ we have 
\begin{align*}
 &|\varphi_{\lambda_{n_k}}(x) - \varphi(x)|\\
 &\leq|\varphi_{\lambda_{n_k}}(x) - \varphi_{\lambda_{n_k}}(x_m)|
 +|\varphi_{\lambda_{n_k}}(x_m) - \varphi(x_m)|
 +|\varphi(x_m)-\varphi(x)|\\
 &\leq 2 C_1\norm{x_m-x}_H + |\varphi_{\lambda_{n_k}}(x_m) - \varphi(x_m)|\,,
\end{align*}
and it is a standard matter to see that this yields 
\[
  \lim_{k\to\infty}\varphi_{\lambda_{n_k}}(x) = \varphi(x) \quad\forall\,x\in H\,.
\]

{\sc Step 2.} Let us consider now the sequence of derivatives $D\varphi_{\lambda_n}:B_j^V\to H$, 
for $j\in\enne_+$ fixed.
Thanks to \eqref{est3} it follows that $(D\varphi_{\lambda_n})_n$ is uniformly 
equicontinuous, while \eqref{est2} implies that 
for every $x\in B^V_r$  the set  $(D\varphi_{\lambda_n}(x))_n$ is relatively compact in $H$.
Again by the Ascoli-Arzel\`a Theorem, for every $j\in\enne$
there exists $\varphi'_j:B^V_j\to H$, continuous with respect to the topology of $H$ on $B^V_j$,
and a subsequence $(\lambda_{n_k}^j)_k$
such that 
\[
  \lim_{k\to\infty}\sup_{x\in B_j^V}\|D\varphi_{\lambda^j_{n_k}}(x)-\varphi'_j(x)\|_H=0\,.
\]
As before, we note that the sequence $(\varphi'_j)_j$ uniquely determines a function $\varphi':V\to H$
such that,  by  possibly using a diagonal argument,
\beq\label{aux_conv2}
  \lim_{k\to\infty}\sup_{x\in B_{r_j}^V}|D\varphi_{\lambda_{n_k}}(x)-\varphi'(x)|=0 
  \quad\forall\,j\in\enne_+\,.
\eeq
In particular, the pointwise convergence 
\[
  \lim_{k\to\infty}\|D\varphi_{\lambda_{n_k}}(x) -\varphi'(x)\|_H=0 \quad\forall\,x\in V
\]
 follows. 
Let us prove that $\varphi'$ uniquely extends to a function in $C^{0,\eta}_b(H;H)$, 
where $\eta\in(0,\frac{1-2\delta}{1+2\delta})$ is fixed.
To this end, thanks to \eqref{est3} we note that for every $x_1,x_2\in V$ we have 
\begin{align*}
  &\|\varphi'(x_1)-\varphi'(x_2)\|_H\\
  &\leq 
  \|\varphi'(x_1) - D\varphi_{\lambda_{n_k}}(x_1)\|_H
  +\|D\varphi_{\lambda_{n_k}}(x_1) - D\varphi_{\lambda_{n_k}}(x_2)\|_H
  +\|D\varphi_{\lambda_{n_k}}(x_2)-\varphi'(x_2)|\\
  &\leq \|\varphi'(x_1) - D\varphi_{\lambda_{n_k}}(x_1)\|_H
  +C_\eta\norm{x_1-x_2}_H^\eta + \|D\varphi_{\lambda_{n_k}}(x_2)-\varphi'(x_2)\|_H
\end{align*}
for every $k\in\enne$, so that letting $k\to\infty$ yields 
\[
  \|\varphi'(x_1)-\varphi'(x_2)\|_H\leq C_\eta\norm{x_1-x_2}_H^\eta \quad\forall\,x_1,x_2\in V\,.
\]
By density of $V$ in $H$ this shows indeed that $\varphi'\in C^{0,\eta}_b(H;H)$, 
where again we have used the same symbol for the extension. 
Furthermore, for every $x\in H$, let $(x_m)_{m\in\enne}\subset V$ be such that 
$x_m\to x$ in $H$ as $m\to\infty$. Then, again by \eqref{est3} and the 
H\"older-continuity of  $ D\varphi_{\lambda_{n_k}}$ and $\varphi'$
 we have 
\begin{align*}
 &\|D\varphi_{\lambda_{n_k}}(x) - \varphi'(x)\|_H\\
 &\leq\|D\varphi_{\lambda_{n_k}}(x) - D\varphi_{\lambda_{n_k}}(x_m)\|_H
 +\|D\varphi_{\lambda_{n_k}}(x_m) - \varphi'(x_m)\|_H
 +\|\varphi'(x_m)-\varphi'(x)\|_H\\
 &\leq 2C_\eta\norm{x_m-x}_H^\eta + \|D\varphi_{\lambda_{n_k}}(x_m) - \varphi'(x_m)\|_H\,,
\end{align*}
so that we deduce the convergence
\[
  \lim_{k\to\infty}\|D\varphi_{\lambda_{n_k}}(x) -\varphi'(x)\|_H=0 \quad\forall\,x\in H\,.
\]

{\sc Step 3.} Let us show that $\varphi\in C^{1,\eta}(H)$ and
$D\varphi=\varphi'$.  By  the convergences \eqref{aux_conv1}--\eqref{aux_conv2} proved above one has that
\[
  (\varphi_{\lambda_{n_k}})_k \quad\text{is Cauchy in } C^1_b(B^V_j) \quad\forall\,j\in\enne_+\,.
\]
It follows in particular that $\varphi$ satisfies $\varphi_{|B^V_j}\in C^1(B^V_j)$
and $D(\varphi_{|B^V_j})=\varphi'_{|B^V_j}$ for every $j$.
In particular, this implies that $\varphi\in C^1(V)$ and 
$D\varphi(x)=\varphi'(x)$ in $V^*$ for every $x\in V$.
Furthermore, let $x,h\in H$ be fixed, and let $(x_m)_m,(h_m)_m\subset V$ such that 
$x_m\to x$ and $h_m\to h$ in $H$ as $m\to\infty$.
Since $\varphi\in C^1(V)$ with $D\varphi=(\varphi')_{|V}:V\to V^*$
and $\varphi'\in C^{0,\eta}(H;H)$, one has that 
\begin{align*}
  \varphi(x_m+ h_m)-\varphi(x_m)&=
  \int_0^1 \ip{D\varphi(x_m+\sigma h_m)}{h_m}_{V^*,V}\,\d \sigma
  =\int_0^1 \left(\varphi'(x_m+\sigma h_m), h_m\right)_{H}\,\d \sigma\,.
\end{align*}
Letting $m\to\infty$ yields, by Dominated Convergence Theorem, 
\[
  \varphi(x+ h)-\varphi(x) = \int_0^1\left(\varphi'(x+\sigma h), h\right)_{H}\,\d \sigma
  \qquad\forall\,x,h\in H\,,
\]
showing that $\varphi:H\to\erre$ is G\^ateaux differentiable with $D\varphi=\varphi':H\to H$.
Since $\varphi'\in C^{0,\eta}_b(H;H)$, this implies then that $\varphi\in C^{1,\eta}_b(H)$,
as required. 
\end{proof}


\section{Uniqueness in law and weak stability}
\label{sec:uniq}

\subsection{Proof of Theorem~\ref{thm2}}
This section is  devoted  to the proof of  our main result
concerning uniqueness in law, namely,  Theorem~\ref{thm2}.
To this end, let $u_0\in H$ be fixed, and let 
\[
  (\Omega, \cF, (\cF_t)_{t\in[0,T]}, \P, W, u_1, y_1,\Lambda) \qquad\text{and}\qquad
  (\Omega, \cF, (\cF_t)_{t\in[0,T]}, \P, W, u_2, y_2,\Lambda)
\]
 be  two Friedrichs-weak solutions to problem \eqref{eq1}--\eqref{eq3}
with respect to the same initial datum $u_0$,
defined on the same stochastic basis, and with  the  same approximating sequence $\Lambda=(\lambda_n)$.
Let $(u_{0,i}^n, G^i_n)_n\subset H\times \cL^2(H,H)$
and $(\Omega, \cF, (\cF_t)_{t\in[0,T]}, \P, W, u_n^i)$, for $i=1,2$, be some sequences 
of data and analytically strong solutions that approximate $u_1$ and $u_2$, respectively,
in the sense of Definition~\ref{def:weak}: in particular, one has for $i=1,2$
and for all $T>0$ that 
\begin{align}
  \label{conv1_n}
  &u_{0,i}^n\to u_0 \quad\text{in } H\,,\\
  \label{conv2_n}
  &G_n^i\to G \quad\text{in } \cL^2(H,H)\,,\\
  \label{conv4_n}
  &u^i_n\wstarto u_i \quad\text{in } L^2_w(\Omega; L^\infty(0,T; H)) \cap L^2_\cP(\Omega; L^2(0,T; V))\,,\\
  \label{conv5_n}
  &u^i_n\to u_i \quad\text{in } L^2(0,T; H)\,, \quad\P\text{-a.s.}\,,\\
  \label{conv6_n}
  &C_{\lambda_n}(u^i_n) \wstarto y_i \quad\text{in } L^\infty_\cP(\Omega\times(0,T); H)\,.
\end{align}

Let  $(\varphi_{\lambda})_{\lambda}$ be the sequence of solutions 
to the regularised Kolmogorov equation \eqref{eq:kolm_reg} as given in 
Proposition~\ref{prop:kolm}, associated to the same sequence $\Lambda$.
Then, Proposition~\ref{prop:lim_kolm} ensures that there exists subsequence $(\lambda_{n_k})_k$
and 
\[
  \varphi\in  C_b^{1, \frac{1-2\delta}{1+2\delta}-}(H) 
  \qquad\text{with}\qquad
  D\varphi\in C^{0}_b(H; D(\L^{ (\frac12-\delta)-})\,,
\] 
depending only on $A$, $B$, $F$, $G$, and $\Lambda$
(but not on $u_1$ and $u_2$), such that, as $k\to\infty$,
 \begin{align}
  \label{conv1_lam}
  \lim_{k\to\infty}|\varphi_{\lambda_{n_k}}(x)-\varphi(x)|=0&\qquad\forall\,x\in H\,,\\
  \label{conv2_lam}
  \lim_{k\to\infty}\|D\varphi_{\lambda_{n_k}}(x)- D\varphi(x)\|_H=0&\qquad\forall\,x\in H\,.
\end{align}

Now, let $i\in\{1,2\}$ be fixed. 
Since $\varphi_{\lambda_{n_k}}\in C^2_b(H)$, It\^o formula for $\varphi_{\lambda_{n_k}}(u_{n_k}^i)$ yields
\begin{align*}
  &\E\varphi_{\lambda_{n_k}}(u_{n_k}^i(t))
  -\frac12\E\int_0^t\operatorname{Tr}\left[G_n^i(G_n^i)^*D^2\varphi_{\lambda_{n_k}}(u_{n_k}^i(s))\right]\,\d s\\
  &\qquad+\E\int_0^t\left(\L u_{n_k}^i(s), D\varphi_{\lambda_{n_k}}(u_{n_k}^i(s))\right)_H\,\d s\\
  &=\varphi_{\lambda_{n_k}}(u_{0,i}^n)
  +\E\int_0^t\left(F(u_{n_k}^i(s)) + \C_{\lambda_{n_k}}(u_{n_k}^i(s)) - f(u_{n_k}^i(s)), 
  D\varphi_{\lambda_{n_k}}(u_{n_k}^i(s))\right)_H\,\d s\,.
\end{align*}
Recalling that $\varphi_{\lambda_{n_k}}$ solves \eqref{eq:kolm_reg}, we deduce that 
\begin{align*}
  &\E\varphi_{\lambda_{n_k}}(u_{n_k}^i(t)) -\alpha\E\int_0^t\varphi_{\lambda_{n_k}}(u_{n_k}^i(s))\,\d s\\
  &\qquad
   +  \frac12\E\int_0^t\operatorname{Tr}
  \left[\left(Q - G_{n_k}^i(G_{n_k}^i)^*\right)D^2\varphi_{\lambda_{n_k}}(u_{n_k}^i(s))\right]\,\d s
  +\E\int_0^tg(u_{n_k}^i(s))\,\d s\\
  &=\varphi_{\lambda_{n_k}}(u_{0,i}^{n_k})
  +\E\int_0^t\left(f_{\lambda_{n_k}}(u_{n_k}^i(s)) - f(u_{n_k}^i(s)), 
  D\varphi_{\lambda_{n_k}}(u_{n_k}^i(s))\right)_H\,\d s
\end{align*}
for every $t\in[0,T]$.
Now, since $f:H\to H$ is maximal monotone, single-valued, and of
bounded range with $D(f)=H$, 
by the Dominated Convergence Theorem as $k\to + \infty$  one has 
that
\[
  f_{\lambda_{n_k}}(u_{n_k}^i) - f(u_{n_k}^i) \to 0 \quad\text{in } L^1_\cP(\Omega; L^2(0,T; H))\,,
\]
which yields in turn, thanks to \eqref{est1}, that
\begin{align*}
  &\E\int_0^t\left(f_{\lambda_{n_k}}(u_{n_k}^i(s)) - f(u_{n_k}^i(s)), 
  D\varphi_{\lambda_{n_k}}(u_{n_k}^i(s))\right)_H\,\d s\\
  &\qquad\leq  C_1
  \norm{f_{\lambda_{n_k}}(u_{n_k}^i) - f(u_{n_k}^i)}_{L^1_\cP(\Omega; L^2(0,T; H))}\to0\,.
\end{align*}
Moreover, given a fixed $\eta\in(0,\frac{1-2\delta}{1+2\delta})$, 
thanks to \eqref{est3},  \eqref{conv1_n}, and \eqref{conv1_lam},
we have
\begin{align*}
  \left|\varphi_{\lambda_{n_k}}(u_{0,i}^{n_k}) - \varphi(u_0)\right|
  &\leq \left|\varphi_{\lambda_{n_k}}(u_{0,i}^{n_k}) - \varphi_{\lambda_{n_k}}(u_0)\right|
  +\left|\varphi_{\lambda_{n_k}}(u_{0}) - \varphi(u_0)\right|\\
  &\leq C\norm{u_{0,i}^{n_k} - u_0}_H^\eta + \left|\varphi_{\lambda_{n_k}}(u_{0}) - \varphi(u_0)\right| \to 0\,.
\end{align*}
Similarly, exploiting again \eqref{est3} we have
\begin{align*}
  \left|\varphi_{\lambda_{n_k}}(u_{n_k}^i (t)) - \varphi(u_i (t))\right|
  &\leq \left|\varphi_{\lambda_{n_k}}(u_{n_k}^i (t)) - \varphi_{\lambda_{n_k}}(u_i (t))\right|
  +\left|\varphi_{\lambda_{n_k}}(u_i (t)) - \varphi(u_i (t))\right|\\
  &\leq C_\eta\norm{u_{n_k}^i (t) - u_i (t)}_H^\eta + \left|\varphi_{\lambda_{n_k}}(u_i (t)) 
  - \varphi(u_i (t))\right|\quad  \forall t\in[0,T]  \,,
\end{align*}
which yields, thanks to \eqref{conv5_n}, \eqref{conv1_lam},
and the Dominated Convergence Theorem, that
\[
  \varphi_{\lambda_{n_k}}(u_{n_k}^i(t)) \to \varphi(u_i(t)) 
  \quad\text{in } L^1(\Omega) \quad\forall\,t\in[0,T]\,.
\]
Eventually, using \eqref{est4} we get for every $s>0$ that 
\begin{align*}
  &\operatorname{Tr}
  \left[\left(Q - G_{n_k}^i(G_{n_k}^i)^*\right)D^2\varphi_{\lambda_{n_k}}(u_{n_k}^i(s))\right]\\
  &\hspace{1.1mm}\stackrel{\eqref{est4}}{\leq}
   C_2\norm{GG^* - G_{n_k}^i(G_{n_k}^i)^*}_{\cL^1(H,H)}\\
  &\quad\leq  C_2\left[\norm{G}_{\cL^2(H,H)}
  \norm{G^*-(G_{n_k}^i)^*}_{\cL^2(H,H)}+
  \norm{G- G_{n_k}^i}_{\cL^2(H,H)}
  \norm{(G_{n_k}^i)^*}_{\cL^2(H,H)}
  \right]\\
  &\quad\leq  C_2\norm{G- G_{n_k}^i}_{\cL^2(H,H)}
  \left(\norm{G}_{\cL^2(H,H)}+
  \norm{G_{n_k}^i}_{\cL^2(H,H)}
  \right)\,.
\end{align*}
Hence, by \eqref{conv2_n} we obtain 
\[
   \left|  \E\int_0^t\operatorname{Tr}
  \left[\left(Q -
      G_{n_k}^i(G_{n_k}^i)^*\right)D^2\varphi_{\lambda_{n_k}}(u_{n_k}^i(s))\right]\,\d
  s \right| 
  \leq  C_2 t\norm{G- G_{n_k}^i}_{\cL^2(H,H)} \to 0\,.
\]
Putting everything together and letting $k\to +  \infty$, for $i=1,2$
 we infer by the arbitrariness of $T>0$ that 
\begin{align*}
  &\E\varphi(u_i(t)) -\alpha\E\int_0^t\varphi(u_i(s))\,\d s
  +\E\int_0^tg(u_i(s))\,\d s =\varphi(u_{0}) \quad\forall\,t\geq0
\end{align*}
An elementary computation yields then 
\[
  e^{-\alpha t}\E\varphi(u_i(t)) 
  +\E\int_0^te^{-\alpha s}g(u_i(s))\,\d s =\varphi(u_{0}) \quad\forall\,t\geq0\,,
\]
so that letting $t\to+\infty$ yields 
\beq
  \label{eq_uniq}
  \varphi(u_{0}) = \E\int_0^{+\infty}e^{-\alpha t}g(u_1(s))\,\d s
  =\E\int_0^{+\infty}e^{-\alpha t}g(u_2(s))\,\d s\,.
\eeq
Since $g$ is arbitrary in $C^{0,s_A}_b(H)$ and $u_1, u_2\in
L^2_\cP(\Omega; C^0([0,+\infty); H))$, 
this implies that $u_1$ and $u_2$ have the same law on
$C^0([0,+\infty); H)$, as desired. For  additional   details,
 the Reader is referred  to 
\cite[sec.~4.1]{BOS_uniq}.

\subsection{Proof of Theorem~\ref{thm3}}
The proof is a direct consequence of the arguments presented in 
the previous section. Indeed, given a sequence $(\mu_n)_n$ as in Theorem~\ref{thm3},
one has that $\mu_n$ is the distribution of an analytically strong solution $u_n$ of 
\eqref{eq:n}. Since $(u_0^n)_n$ are bounded in $H$ and 
$(G_n)_n$ are bounded in $\cL^2(H,H)$, the same arguments of Section~\ref{sec:exist}
ensure that $(u_n)_n$ is bounded in $L^1_\mP(\Omega; C^0([0,T]; H)\cap L^2(0,T; V))$, 
hence that $(\mu_n)_n$ is tight in $L^2(0,T; H)$.
We infer that, possibly extracting a subsequence, one has that 
$\mu_n\wstarto\mu$ in $L^2(0,T; H)$. Moreover, the proof of Theorem~\ref{thm2}
also ensures that such $\mu$ is unique and coincides with 
the law of $u$, where $u$ is a Friedrichs-weak solution in the sense of 
Definition~{def:weak}:
hence, the convergence 
$\mu_n\wstarto\mu$ hols for the entire sequence $(\mu_n)_n$, and this concludes the proof.


\section*{Acknowledgement}
This research was funded in
whole or in part by the Austrian Science Fund (FWF) projects 10.55776/F65,  10.55776/I5149,
10.55776/P32788, 
as well as by the OeAD-WTZ project CZ 09/2023. 
C.O.~and L.S.~are members of Gruppo Nazionale 
per l'Analisi Matematica, la Probabilit\`a 
e le loro Applicazioni (GNAMPA), 
Istituto Nazionale di Alta Matematica (INdAM).
The present research is also part of the activities of 
``Dipartimento di Eccellenza 2023-2027'' of Politecnico di Milano.
For open access purposes, the authors have applied a CC BY public copyright
license to any author-accepted manuscript version arising from this
submission.

\section*{Data availability statement}
No new data were created or analysed in this study. 
Data sharing is not applicable.

\section*{Conflict of interest statement}
The authors have no conflicts of interest to declare.


\bibliographystyle{abbrv}


\end{document}